\documentclass[11pt]{amsart}

\usepackage{amsmath}
\usepackage{amssymb}
\usepackage{amscd}
\usepackage{mathrsfs}
\usepackage{color}
\usepackage{hyperref}

\usepackage{xy}
\xyoption{all}

\newcommand{\nc}{\newcommand}

\nc{\red}[1]{\textcolor{red}{#1}}

\nc{\exto}[1]{\stackrel{#1}{\longrightarrow}}
\nc{\dlim}{{\mathop{\lim\limits_{\longrightarrow}\,}}}
\nc{\ilim}{{\mathop{\lim\limits_{\longleftarrow}\,}}}
\nc{\hocolim}{{\mathop{\sf hocolim}\,}}
\nc{\holim}{{\mathop{\sf holim}}}
\nc{\lan}{\big\langle}
\nc{\ran}{\big\rangle}

\nc{\kk}{{\mathsf{k}}}

\nc{\C}{{\mathbb{C}}}
\nc{\HH}{{\mathbf{H}}}
\nc{\DD}{{\mathbb{D}}}
\nc{\LL}{{\mathbb{L}}}
\nc{\PP}{{\mathbb{P}}}
\nc{\QQ}{{\mathbb{Q}}}
\nc{\RR}{{\mathbb{R}}}
\nc{\ZZ}{{\mathbb{Z}}}

\nc{\CA}{{\mathcal{A}}}
\nc{\CB}{{\mathcal{B}}}
\nc{\CC}{{\mathcal{C}}}
\nc{\D}{{\mathcal{D}}}
\nc{\SCA}{{\mathscr{A}}}
\nc{\SCB}{{\mathscr{B}}}
\nc{\SCC}{{\mathscr{C}}}
\nc{\SCD}{{\mathscr{D}}}
\nc{\SCE}{{\mathscr{E}}}
\nc{\SCH}{{\mathscr{H}}}
\nc{\CE}{{\mathcal{E}}}
\nc{\CF}{{\mathcal{F}}}
\nc{\CG}{{\mathcal{G}}}
\nc{\CH}{{\mathcal{H}}}
\nc{\CK}{{\mathscr{K}}}
\nc{\CL}{{\mathcal{L}}}
\nc{\CM}{{\mathcal{M}}}
\nc{\CN}{{\mathcal{N}}}
\nc{\CO}{{\mathcal{O}}}
\nc{\CQ}{{\mathcal{Q}}}
\nc{\CR}{{\mathcal{R}}}
\nc{\CS}{{\mathcal{S}}}
\nc{\CT}{{\mathcal{T}}}
\nc{\CU}{{\mathcal{U}}}
\nc{\CV}{{\mathscr{V}}}
\nc{\CW}{{\mathcal{W}}}
\nc{\CX}{{\mathcal{X}}}
\nc{\CY}{{\mathcal{Y}}}
\nc{\CMo}{{\mathcal{M}^\circ}}
\nc{\Co}{{{C}^\circ}}

\nc{\BY}{{\overline{Y}}}
\nc{\BYD}{{\overline{Y}{}^{|D|}}}
\nc{\OZ}{{\overline{Z}}}
\nc{\bg}{{\bar{g}}}

\nc{\ba}{{\mathbf{a}}}
\nc{\bb}{{\mathbf{b}}}
\nc{\be}{{\mathbf{e}}}
\nc{\bq}{{\mathbf{q}}}
\nc{\bx}{{\mathbf{x}}}
\nc{\by}{{\mathbf{y}}}
\nc{\bz}{{\mathbf{z}}}
\nc{\BB}{{\mathbf{B}}}
\nc{\BC}{{\mathbf{C}}}
\nc{\BE}{{\mathbf{E}}}
\nc{\BD}{{\mathbf{D}}}
\nc{\BG}{{\mathbf{G}}}
\nc{\BM}{{\mathbf{M}}}
\nc{\BP}{{\mathbf{P}}}
\nc{\bp}{{\mathbf{p}}}
\nc{\BU}{{\mathbf{U}}}
\nc{\BZ}{{\mathbf{Z}}}
\nc{\BPr}{{\mathsf{P}}}
\nc{\BL}{{\mathbf{L}}}
\nc{\BR}{{\mathbf{R}}}
\nc{\BRO}[1]{{{\mathbf{R}}^{\circ}_{#1}}}
\nc{\BRD}[1]{{{\mathbf{R}}^{|D|}_{#1}}}
\nc{\BRP}[1]{{{\mathbf{R}}^{1}_{#1}}}
\nc{\BRTP}[1]{{{\mathbf{\tilde{R}}}{}^{1}_{#1}}}
\nc{\BS}{{\mathbf{S}}}
\nc{\BMS}{{{\mathbf{M}}^{{s}}}}
\nc{\BMSS}{{{\mathbf{M}}^{{ss}}}}
\nc{\BMZ}{{\mathbf{M}^{\circ}}}
\nc{\BCL}{{\mathbf{L}}}

\nc{\PCC}{{{}^\perp\CC}}

\nc{\Cl}{{\mathsf{Cliff}}}
\nc{\Clev}{{\mathop{\mathsf{Cliff}}^{\circ}}}

\nc{\FA}{{\mathfrak{A}}}
\nc{\FB}{{\mathfrak{B}}}
\nc{\FU}{{\mathfrak{U}}}

\nc{\fa}{{\mathfrak{a}}}
\nc{\fb}{{\mathfrak{b}}}
\nc{\fg}{{\mathfrak{g}}}
\nc{\fn}{{\mathfrak{n}}}
\nc{\fp}{{\mathfrak{p}}}
\nc{\FD}{{\mathfrak{D}}}
\nc{\FE}{{\mathfrak{E}}}
\nc{\FL}{{\mathfrak{L}}}
\nc{\FM}{{\mathfrak{M}}}
\nc{\FS}{{\mathsf{S}}}

\nc{\se}{{\mathsf{e}}}
\nc{\sfc}{{\mathsf{c}}}
\nc{\sfch}{{\mathsf{ch}}}
\nc{\sfh}{{\mathsf{h}}}

\nc{\SK}{{\mathsf{K}}}
\nc{\SM}{{\mathsf{M}}}
\nc{\SO}{{\mathsf{O}}}
\nc{\SQ}{{\mathsf{Q}}}
\renewcommand{\SS}{{\mathsf{S}}}
\nc{\SPV}{{\mathsf{S}^+\mathsf{V}}}
\nc{\SMV}{{\mathsf{S}^-\mathsf{V}}}
\nc{\SPMV}{{\mathsf{S}^\pm\mathsf{V}}}
\nc{\SX}{{S_X}}
\nc{\SY}{{S_Y}}
\nc{\phipsi}{{q}}
\nc{\eps}{\varepsilon}

\nc{\pim}{{\pi_-}}
\nc{\pip}{{\pi_+}}

\nc{\TE}{{\tilde{\CE}}}
\nc{\TQ}{{\tilde{Q}}}
\nc{\TCF}{{\tilde{\CF}}}
\nc{\TCG}{{\tilde{\CG}}}
\nc{\TCL}{{\tilde{\CL}}}
\nc{\TF}{{\tilde{F}}}
\nc{\TW}{{\tilde{W}}}
\nc{\TCC}{{\tilde{\CC}}}
\nc{\TCX}{{\tilde{\CX}}}
\nc{\TCY}{{\tilde{\CY}}}
\nc{\TPhi}{{\tilde{\Phi}}}
\nc{\OPhi}{{\bar{\Phi}}}
\nc{\txi}{{\tilde{\xi}}}
\nc{\tp}{{\tilde{p}}}
\nc{\tq}{{\tilde{q}}}
\nc{\tzeta}{{\tilde{\zeta}}}
\nc{\tpi}{{\tilde{\pi}}}

\nc{\halpha}{{\hat{\alpha}}}
\nc{\HCA}{{\hat{\CA}}}
\nc{\HCB}{{\hat{\CB}}}
\nc{\HCC}{{\hat{\CC}}}
\nc{\HE}{{\widehat{\CE}}}
\nc{\HX}{{\hat{X}}}
\nc{\hxi}{{\hat{\xi}}}

\nc{\UH}{{\mathcal{H}}}

\nc{\TM}{{\widetilde{M}}}
\nc{\TCM}{{\widetilde{\CM}}}
\nc{\TU}{{\widetilde{U}}}
\nc{\TX}{{\widetilde{X}}}
\nc{\TY}{{\widetilde{Y}}}
\nc{\TYO}{{{\widetilde{Y}}^\circ}}
\nc{\barf}{{\bar{f}}}
\nc{\te}{{\tilde{e}}{}}
\nc{\tf}{{\tilde{f}}}
\nc{\tg}{{\tilde{g}}}
\nc{\ti}{{\tilde{\imath}}}
\nc{\tj}{{\tilde{\jmath}}}
\nc{\ty}{{\tilde{y}}}
\nc{\tphi}{{\tilde{\phi}}}

\nc{\urho}{{\underline{\rho}}}

\nc{\LRA}{\Leftrightarrow}
\nc{\RA}{\Rightarrow}
\nc{\lotimes}{\mathbin{\mathop{\otimes}\limits^{\mathbb{L}}}}
\nc{\CEnd}{\mathop{\mathcal{E}\mathit{nd}}\nolimits}
\nc{\CExt}{\mathop{\mathcal{E}\mathit{xt}}\nolimits}
\nc{\CHom}{\mathop{\mathscr{H}\!\!\mathit{om}}\nolimits}
\nc{\RH}{\mathop{{\mathsf{R}}\Gamma}\nolimits}
\nc{\RGamma}{\mathop{{\mathsf{R}}\Gamma}\nolimits}
\nc{\RHom}{\mathop{\mathsf{RHom}}\nolimits}
\nc{\RCHom}{\mathop{\mathsf{R}\!\!\CHom}\nolimits}
\nc{\RG}{\mathop{\mathsf{R\Gamma}}\nolimits}
\nc{\Hom}{\mathop{\mathsf{Hom}}\nolimits}
\nc{\Ext}{\mathop{\mathsf{Ext}}\nolimits}
\nc{\End}{\mathop{\mathsf{End}}\nolimits}
\nc{\Tor}{\mathop{\mathsf{Tor}}\nolimits}
\nc{\Tordim}{\mathop{\mathsf{Tor}\text{\rm-}\mathsf{dim}}\nolimits}
\nc{\Hilb}{\mathop{\mathsf{Hilb}}\nolimits}
\nc{\Spec}{\mathop{\mathsf{Spec}}\nolimits}
\nc{\Pic}{\mathop{\mathsf{Pic}}\nolimits}
\renewcommand{\Im}{\mathop{\mathsf{Im}}\nolimits}
\nc{\Tr}{\mathop{\mathsf{Tr}}\nolimits}
\nc{\tr}{\mathop{\mathsf{tr}}\nolimits}
\nc{\sd}{\mathop{\mathsf{sd}}\nolimits}
\nc{\LInd}{\mathop{\mathsf{LInd}}\nolimits}
\nc{\Res}{\mathop{\mathsf{Res}}\nolimits}
\nc{\Cone}{\mathop{\mathsf{Cone}}\nolimits}
\nc{\Fiber}{\mathop{\mathsf{Fiber}}\nolimits}
\nc{\Ker}{\mathop{\mathsf{Ker}}\nolimits}
\nc{\Coker}{\mathop{\mathsf{Coker}}\nolimits}
\nc{\codim}{\mathop{\mathsf{codim}}\nolimits}
\nc{\sing}{{\mathsf{sing}}}
\nc{\supp}{\mathop{\mathsf{supp}}}
\nc{\vol}{\mathop{\mathsf{vol}}\nolimits}
\nc{\ch}{\mathop{\mathsf{ch}}\nolimits}
\nc{\perf}{{\mathsf{perf}}}
\nc{\rank}{\mathop{\mathsf{rank}}}
\nc{\Pf}{{\mathsf{Pf}}}
\nc{\Gr}{{\mathsf{Gr}}}
\nc{\OGr}{{\mathsf{OGr}}}
\nc{\Flag}{{\mathsf{Fl}}}
\nc{\Kosz}{{\mathsf{Kosz}}}
\nc{\LGr}{{\mathsf{LGr}}}
\nc{\GTGr}{{\mathsf{G_2Gr}}}
\nc{\GTF}{{\mathsf{G_2F}}}
\nc{\OF}{{\mathsf{OF}}}
\nc{\Fl}{{\mathsf{Fl}}}
\nc{\Bl}{{\mathsf{Bl}}}
\nc{\Br}{{\mathsf{Br}}}
\nc{\Fr}{{\mathsf{Fr}}}
\nc{\GL}{{\mathsf{GL}}}
\nc{\PGL}{{\mathsf{PGL}}}
\nc{\SL}{{\mathsf{SL}}}
\nc{\SP}{{\mathsf{Sp}}}
\nc{\Spin}{{\mathsf{Spin}}}
\nc{\Tot}{{\mathsf{Tot}}}
\nc{\ev}{{\mathsf{ev}}}
\nc{\od}{{\mathsf{odd}}}
\nc{\coev}{{\mathsf{coev}}}
\nc{\id}{{\mathsf{id}}}
\nc{\opp}{{\mathsf{opp}}}
\nc{\PS}{{{\PP^3}}}
\nc{\Qu}{{{Q^3}}}
\nc{\tdim}{\mathop{\Tor\dim}}
\nc{\gdim}{\mathop{\mathrm{gdim}}}
\nc{\cydim}{\mathop{\dim^{{\mathrm{CY}}}}}
\nc{\ecart}{{\fbox{$\scriptstyle\mathsf{EC}$}}}
\nc{\ad}{{\mathop{\mathsf ad}}}
\nc{\sg}{{\mathop{\mathsf sg}}}
\nc{\hf}{{\mathop{\mathsf hf}}}
\nc{\gr}{{\mathop{\mathsf gr}}}
\nc{\qgr}{{\mathop{\mathsf qgr}}}
\nc{\Perf}{{\mathop{{\mathsf{Perf}}}}}
\nc{\Coh}{{\mathop{{\mathsf{Coh}}}}}
\nc{\Com}{{\mathop{{\mathsf{Com}}}}}
\nc{\Qis}{{\mathop{{\mathsf{Qiso}}}}}
\nc{\coh}{{\mathop{{\mathsf{coh}}}}}
\nc{\qcoh}{{\mathop{{\mathsf{qcoh}}}}}
\nc{\dgm}{{\mathop{{\mathsf{dgm}\text{-}}}}}
\nc{\acy}{{\mathop{{\mathsf{acycl}\text{-}}}}}
\nc{\hproj}{{\mathop{{\mathsf{hproj}\text{-}}}}}
\nc{\com}{{\mathop{{\mathsf{com}}}}}
\nc{\comp}{{\mathop{{\mathsf{comp}}}}}
\nc{\Ab}{{\mathop{\mathcal{A}\mathit{b}}}}
\nc{\Open}{{\mathop{\mathsf{Open}}}}
\nc{\Ccoh}{{\mathop{\mathsf Ccoh}}}
\nc{\Qcoh}{{\mathop{\mathsf{Qcoh}}}}
\nc{\arcs}{{\mathop{\mathsf{arcs}}}}
\nc{\Shuf}{{\mathop{\mathsf{Shuf}}}}
\nc{\circles}{{\mathop{\mathsf{circles}}}}
\nc{\normcircles}{\underline{\mathop{\mathsf{circles}}}}
\nc{\At}{{\mathop{\mathsf{At}}\nolimits}}
\nc{\tra}{{\mathsf{T}}}
\nc{\fsl}{{\mathfrak{sl}}}
\nc{\fso}{{\mathfrak{so}}}
\nc{\fgl}{{\mathfrak{gl}}}
\nc{\Griff}{{\mathop{\mathsf{Griff}}}}

\nc{\AAV}{{\mathcal{AAV}}}

\nc{\Rep}{{\mathsf{Rep}}}

\nc{\Cubics}{{{\mathcal{S}}_3}}
\nc{\VFT}{{{\mathcal{S}}_{14}}}
\nc{\VFTE}{{{\mathcal{N}}_{\mathrm{reg,sm}}}}
\nc{\MX}{{\CM_X}}
\nc{\MY}{{\CM_Y}}
\nc{\MYE}{{\CM_{Y,\CE}}}
\nc{\Yd}{{Y_d}}
\nc{\Yfive}{{Y_5}}
\nc{\Xg}{{X_{2g-2}}}
\nc{\Xtt}{{X_{22}}}
\nc{\Xst}{{X_{16}}}
\nc{\Xtw}{{X_{12}}}
\nc{\Xe}{{X_{8}}}
\nc{\Xf}{{X_{4}}}

\nc{\git}{{/\!\!/\!{}_\chi}}

\nc{\HOH}{{\mathsf H\mathsf H}}
\nc{\NHH}{{\mathsf N\HOH}}
\nc{\HHE}{{\mathsf H\mathsf E}}

\nc{\he}{{\mathrm{h}}}
\nc{\ph}{{\mathrm{ph}}}
\nc{\ac}{{\mathrm{ac}}}

\theoremstyle{plain}

\newtheorem{theorem}{Theorem}[section]
\newtheorem{conjecture}[theorem]{Conjecture}
\newtheorem{lemma}[theorem]{Lemma}
\newtheorem{proposition}[theorem]{Proposition}
\newtheorem{corollary}[theorem]{Corollary}

\theoremstyle{definition}

\newtheorem{definition}[theorem]{Definition}

\newtheorem{exercise}[theorem]{Exercise}
\newtheorem{example}[theorem]{Example}

\theoremstyle{remark}

\newtheorem{remark}[theorem]{Remark}

\title{Derived categories view on rationality problems}
\author{Alexander Kuznetsov}
\dedicatory{Lecture notes for the CIME--CIRM summer school, Levico Terme, June 22--27, 2015}
\address{\sloppy
\parbox{0.9\textwidth}{
Algebraic Geometry Section, Steklov Math Institute,
8 Gubkin str., Moscow 119991 Russia
\hfill\\[5pt]
The Poncelet Laboratory, Independent University of Moscow
\hfill\\[5pt]
Laboratory of Algebraic Geometry, SU-HSE
\hfill
}\bigskip}
\email{akuznet@mi.ras.ru}

\date{}
\thanks{I was partially supported by
by RFBR grants 14-01-00416, 15-01-02164, 15-51-50045 and NSh-2998.2014.1, and
a subsidy granted to the HSE by the Government of the Russian Federation 
for the implementation of the Global Competitiveness Program.}

\begin{document}

\begin{abstract}
We discuss a relation between the structure of derived categories of smooth projective varieties
and their birational properties. We suggest a possible definition of a birational invariant,
the derived category analogue of the intermediate Jacobian, and discuss its possible applications 
to the geometry of prime Fano threefolds and cubic fourfolds.
\end{abstract}

\maketitle





These lectures were prepared for the summer school ``Rationality Problems in Algebraic Geometry'' 
organized by CIME-CIRM in Levico Terme in June~2015.
I would like to thank Rita Pardini and Pietro Pirola (the organizers of the school)
and all the participants for the wonderful atmosphere.

\section{Introduction into derived categories and semiorthogonal decompositions}

Derived categories were defined by Verdier in his thesis~\cite{Ve} back in 60's.
When appeared they were used as an abstract notion to formulate general results, 
like Grothendieck--Riemann--Roch theorem (for which actually they were devised by Grothendieck).
Later on, they were actively used as a technical tool, see e.g.~\cite{Ha}.
The situation changed with appearance of Beilinson's brilliant paper~\cite{Bei},
when they attracted attention as the objects of investigation. Finally, results
of Bondal and Orlov \cite{BO95,BO02} put them on their present place in the center of algebraic geometry.

\subsection{Derived categories}

We refer to~\cite{GM} for a classical treatment of derived and triangulated categories,
and to~\cite{H06} for a more geometrically oriented discussion. Here we restrict ourselves 
to give a very brief introduction into the subject.

Let $\kk$ be a base field.
Recall that a complex over a $\kk$-linear abelian category $\CA$ is a collection of objects $F^i \in \CA$ and morphisms
$d_F^i:F^i \to F^{i+1}$ such that $d_F^{i+1}\circ d_F^i = 0$ for all $i \in \ZZ$. A~complex is bounded if $F^i = 0$ for all $|i| \gg0$.
A~morphism of complexes
$(F^\bullet,d_F^\bullet) \to (G^\bullet,d_G^\bullet)$ is a collection of morphisms $\varphi^i:F^i \to G^i$
in $\CA$ commuting with the differentials: $d_G^i\circ\varphi^i = \varphi^{i+1}\circ d_F^i$ for all $i \in \ZZ$.
The category of bounded complexes in $\CA$ is denoted by $\Com^b(\CA)$.

The $i$-th cohomology of a complex $(F^\bullet,d_F^\bullet)$ is an object of the abelian category $\CA$ defined by
\begin{equation*}
\CH^i(F) = \frac{\Ker(d_F^i:F^i \to F^{i+1})}{\Im(d_F^{i-1}:F^{i-1} \to F^i)}.
\end{equation*}
Clearly, the cohomology is a functor $\CH^i:\Com^b(\CA) \to \CA$. In particular, a morphism of complexes 
$\varphi:F^\bullet \to G^\bullet$ induces a morphism of cohomology objects $\CH^i(\varphi):\CH^i(F^\bullet) \to \CH^i(G^\bullet)$.
A morphism $\varphi:F^\bullet \to G^\bullet$ of complexes is called a {\sf quasiisomorphism} if for all $i \in \ZZ$ 
the morphism $\CH^i(\varphi)$ is an isomorphism.

\begin{definition}
The {\sf bounded derived category} of an abelian category $\CA$ is the localization
\begin{equation*}
\BD^b(\CA) = \Com^b(\CA)[\Qis^{-1}].
\end{equation*}
of the category of bounded complexes in $\CA$ with respect to the class of quasiisomorphisms.
\end{definition}

Of course, this definition itself requires an explanation, which we prefer to skip since there are many textbooks (e.g.~\cite{GM})
describing this in detail. Here we just restrict ourselves by saying that a localization of a category $\CC$
in a class of morphisms $S$ is a category $\CC[S^{-1}]$ with a functor $\CC \to \CC[S^{-1}]$ such that the images
of all morphisms in $S$ under this functor are invertible, and which is the smallest with this property (that is enjoys
a universal property). 

Sometimes it is more convenient to consider the unbounded version of the derived category, but we will not need it.



The cohomology functors descend to the derived category, so
$\CH^i:\BD^b(\CA) \to \CA$
is an additive functor for each $i \in \ZZ$. 
Further, there is a full and faithful embedding functor
\begin{equation*}
\CA \to \BD^b(\CA),
\end{equation*}
taking an object $F \in \CA$ to the complex $\dots \to 0 \to F \to 0 \to \dots$ with zeroes everywhere 
outside of degree zero, in which the object $F$ sits. This complex has only one nontrivial cohomology
which lives in degree zero and equals $F$.


\subsection{Triangulated categories}

The most important structure on the derived category is the triangulated structure.

\begin{definition}
A triangulated category is an additive category $\CT$ equipped with
\begin{itemize}
\item an automorphism of $\CT$ called {\sf the shift functor} and denoted by $[1]:\CT \to \CT$, 
the powers of the shift functor are denoted by $[k]:\CT \to \CT$ for all $k \in \ZZ$;
\item a class 
of chains of morphisms in $\CT$ of the form
\begin{equation}\label{dtdef}
F_1 \xrightarrow{\ \varphi_1\ } F_2 \xrightarrow{\ \varphi_2\ } F_3 \xrightarrow{\ \varphi_3\ } F_1[1]
\end{equation}
called {\sf distinguished triangles},
\end{itemize}
which satisfy a number of axioms (see~\cite{GM}). 
\end{definition}

Instead of listing all of them we will discuss only the most important
axioms and properties.

First, each morphism $F_1 \xrightarrow{\ \varphi_1\ } F_2$ can be extended to a distinguished triangle~\eqref{dtdef}.
The extension is unique up to a noncanonical isomorphism, the third vertex of such a triangle is called {\sf a cone} 
of the morphism $\varphi_1$ and is denoted by $\Cone(\varphi_1)$. 

Further, a triangle~\eqref{dtdef} is distinguished if and only if the triangle
\begin{equation}\label{dtrot}
F_2 \xrightarrow{\ \varphi_2\ } F_3 \xrightarrow{\ \varphi_3\ } F_1[1] \xrightarrow{\ \varphi_1[1]\ } F_2[1]
\end{equation}
is distinguished. Such triangle is referred to as the {\sf rotation} of the original triangle. Clearly, rotating 
a distinguished triangle in both directions, one obtains an infinite chain of morphisms
\begin{equation*}
\dots \xrightarrow{\ \ } F_3[-1] \xrightarrow{\ \varphi_3[-1]\ } 
F_1 \xrightarrow{\ \varphi_1\ } F_2 \xrightarrow{\ \varphi_2\ } F_3 \xrightarrow{\ \varphi_3\ } 
F_1[1] \xrightarrow{\ \varphi_1[1]\ } F_2[1] \xrightarrow{\ \varphi_2[1]\ } F_3[1] \xrightarrow{\ \varphi_3[1]\ } F_1[2] \xrightarrow{\ \ } \dots,
\end{equation*}
called a {\sf helix}. Any consecutive triple of morphisms in a helix is thus a distinguished triangle.

Finally, the sequence of $\kk$-vector spaces 
\begin{equation*}
\dots \xrightarrow{\,\,} \Hom(G,F_3[-1]) \xrightarrow{\,\varphi_3[-1]\,} 
\Hom(G,F_1) \xrightarrow{\,\varphi_1\,} \Hom(G,F_2) \xrightarrow{\,\varphi_2\,} \Hom(G,F_3) \xrightarrow{\,\varphi_3\,} 
\Hom(G,F_1[1])
\xrightarrow{\,\,} \dots,
\end{equation*}
obtained by applying the functor $\Hom(G,-)$ to a helix, is a long exact sequence.
Analogously, the sequence of $\kk$-vector spaces 
\begin{equation*}
\dots \xrightarrow{\,\,} \Hom(F_1[1],G) \xrightarrow{\,\varphi_3\,} 
\Hom(F_3,G) \xrightarrow{\,\varphi_2\,} \Hom(F_2,G) \xrightarrow{\,\varphi_1\,} \Hom(F_1,G) \xrightarrow{\,\varphi_3[-1]\,} 
\Hom(F_3[-1],G)
\xrightarrow{\,\,} \dots,
\end{equation*}
obtained by applying the functor $\Hom(-,G)$ to a helix, is a long exact sequence.

\begin{exercise}
Assume that in a distinguished triangle~\eqref{dtdef} one has $\varphi_3 = 0$. Show that there is an isomorphism
$F_2 \cong F_1 \oplus F_3$ such that $\varphi_1$ is the embedding of the first summand, and $\varphi_2$ is the projection
onto the second summand.
\end{exercise}

\begin{exercise}
Show that a morphism $\varphi_1 : F_1 \to F_2$ is an isomorphism if and only if $\Cone(\varphi_1) = 0$.
\end{exercise}

Derived category $\BD^b(\CA)$ carries a natural triangulated structure.
The shift functor is defined by
\begin{equation}\label{shiftda}
(F[1])^i = F^{i+1},\qquad
d_{F[1]}^i = -d_F^{i+1},\qquad
(\varphi[1])^i = - \varphi^{i+1}.
\end{equation}
The cone of a morphism of complexes $\varphi:F \to G$ is defined by
\begin{equation}\label{coneda}
\Cone(\varphi)^i = G^i \oplus F^{i+1},
\qquad
d_{\Cone(\varphi)}(g^i,f^{i+1}) = (d_G(g^i) + \varphi(f^{i+1}), -d_F(f^{i+1})).
\end{equation} 
A morphism $F \xrightarrow{\ \varphi\ } G$ extends to a triangle by morphisms
\begin{equation*}
\epsilon:G \to \Cone(\varphi),\quad \epsilon(g^i) = (g^i,0),
\qquad
\rho:\Cone(\varphi) \to F[1],\quad \rho(g^i,f^{i+1}) = f^{i+1}.
\end{equation*}
One defines a distinguished triangle in $\BD^b(\CA)$ as a triangle isomorphic to the triangle
\begin{equation*}
F \xrightarrow{\ \varphi\ } G \xrightarrow{\ \epsilon\ } \Cone(\varphi) \xrightarrow{\ \rho\ } F[1]
\end{equation*}
defined above.

\begin{theorem}[\cite{Ve}]
The shift functor and the above class of distinguished triangles provide $\BD^b(\CA)$ with a structure of a triangulated category.
\end{theorem}

For $F,G \in \CA$ the spaces of morphisms in the derived category between $F$ and shifts of $G$ are identified with the $\Ext$-groups
in the original abelian category
\begin{equation*}
\Hom(F,G[i]) = \Ext^i(F,G).
\end{equation*}
For arbitrary objects $F,G \in \BD^b(\CA)$ we use the left hand side of the above equality as the definition of the right hand side.
With this convention 
%
the long exact sequences obtained by applying $\Hom$ functors to a helix can be rewritten as
\begin{align*}
\dots \xrightarrow{\,\,} \Ext^{i-1}(G,F_3) \xrightarrow{\,\,} 
\Ext^i(G,F_1) \xrightarrow{\,\,} \Ext^i(G,F_2) \xrightarrow{\,\,} \Ext^i(G,F_3) \xrightarrow{\,\,} 
\Ext^{i+1}(G,F_1) \xrightarrow{\,\,} \dots\\
\dots \xrightarrow{\,\,} \Ext^{i-1}(F_1,G) \xrightarrow{\,\,} 
\Ext^i(F_3,G) \xrightarrow{\,\,} \Ext^i(F_2,G) \xrightarrow{\,\,} \Ext^i(F_1,G) \xrightarrow{\,\,} 
\Ext^{i+1}(F_3,G) \xrightarrow{\,\,} \dots
\end{align*}

The most important triangulated category for the geometry of an algebraic variety $X$ is the bounded derived category 
of coherent sheaves on $X$. To make notation more simple we use the following shorthand
\begin{equation*}
\BD(X) := \BD^b(\coh(X)).
\end{equation*}
Although most of the results we will discuss are valid in a much larger generality,
we restrict for simplicity to the case of smooth projective varieties. Sometimes one also may need
some assumptions on the base field, so let us assume for simplicity that $\kk = \C$.

\subsection{Functors}

A {\sf triangulated functor} between triangulated categories $\CT_1 \to \CT_2$ is a pair $(\Phi,\phi)$,
where
$\Phi:\CT_1 \to \CT_2$ is a $\kk$-linear functor $\CT_1 \to \CT_2$, 
which takes distinguished triangles of $\CT_1$ to distinguished triangles of $\CT_2$, and
$\phi:\Phi\circ[1]_{\CT_1} \to [1]_{\CT_2}\circ\Phi$ is an isomorphism of functors.
Usually the isomorphism $\phi$ will be left implicit.

There is a powerful machinery (see~\cite{GM} for the classical or~\cite{Ke06} for the modern approach) which allows to extend an additive functor between abelian
categories to a triangulated functor between their derived categories. Typically, if the initial functor 
is right exact, one extends it as the left derived functor (by applying the original functor
to a resolution of the object of projective type), and if the initial functor is left exact, one extends
it as the right derived functor (by using a resolution of injective type). We do not stop here 
on this technique. Instead, we list the most important (from the geometric point of view)
triangulated functors between derived categories of coherent sheaves (see~\cite{Ha}, or~\cite{H06} for more details).


\subsubsection{Pullbacks and pushforwards}
Let $f:X \to Y$ be a morphism of (smooth, projective) schemes. It gives an adjoint pair of functors $(f^*,f_*)$, where 
\begin{itemize}
\item $f^*:\coh(Y) \to \coh(X)$ is the pullback functor, and
\item $f_*:\coh(X) \to \coh(Y)$ is the pushforward functor.
\end{itemize}
This adjoint pair induces an adjoint
pair of derived functors $(Lf^*,Rf_*)$ on the derived categories, where
\begin{itemize}
\item $Lf^*:\BD(Y) \to \BD(X)$ is the (left) derived pullback functor, and
\item $Rf_*:\BD(X) \to \BD(Y)$ is the (right) derived pushforward functor.
\end{itemize}
%
%
The cohomology sheaves of the derived pullback/pushforward applied to a coherent sheaf $\CF$ are well known as the classical {\sf higher pullbacks/pushforwards}
\begin{equation*}
L_if^*(\CF) = \CH^{-i}(Lf^*\CF),
\qquad
R^if_*(\CF) = \CH^i(Rf_*\CF).
\end{equation*}
When $f:X \to \Spec(\kk)$ is the structure morphim of a $\kk$-scheme $X$, the pushforward functor $f_*$
is identified with the global sections functor $\Gamma(X,-)$, so that $R^if_* = H^i(X,-)$ and $Rf_* \cong R\Gamma(X,-)$.

\subsubsection{Twisted pullbacks}
The derived pushforward functor $Rf_*$ also has a right adjoint functor
\begin{equation*}
f^!:\BD(Y) \to \BD(X),
\end{equation*}
which is called sometimes {\sf the twisted pullback functor}. The adjunction of $Rf_*$ and $f^!$ is known
as {\sf the Grothendieck duality}.
The twisted pullback $f^!$ has a very simple relation with the derived pullback functor (under our assumption of smoothness and projectivity)
\begin{equation*}
f^!(F) \cong Lf^*(F) \otimes \omega_{X/Y}[\dim X - \dim Y],
\end{equation*}
where 
$\omega_{X/Y} = \omega_X\otimes f^*\omega_Y^{-1}$,
is the relative dualizing sheaf.

\subsubsection{Tensor products and local $\CHom$'s}

Another important adjoint pair of functors on the category $\coh(X)$ of coherent sheaves is $(\otimes,\CHom)$. 
In fact these are bifunctors (each has two arguments), and the adjunction is a functorial isomorphism
\begin{equation*}
\Hom(\CF_1\otimes\CF_2,\CF_3) \cong \Hom(\CF_1,\CHom(\CF_2,\CF_3)).
\end{equation*}
This adjoint pair induces an adjoint pair of derived functors $(\lotimes,\RCHom)$ on the derived categories, where
\begin{itemize}
\item $\lotimes:\BD(X)\times\BD(X) \to \BD(X)$ is the (left) derived tensor product functor, and
\item $\RCHom:\BD(X)^\opp\times\BD(X) \to \BD(X)$ is the (right) derived local $\CHom$ functor.
\end{itemize}
%
%
The cohomology sheaves of these functors applied to a pair of coherent sheaves $\CF,\CG \in \coh(X)$ are the classical $\Tor$'s and $\CExt$'s:
\begin{equation*}
\Tor_i(\CF,\CG) = \CH^{-i}(\CF\lotimes\CG),
\qquad
\CExt^i(\CF,\CG) = \CH^i(\RCHom(\CF,\CG)).
\end{equation*}
One special case of the $\RCHom$ functor is very useful. The object
\begin{equation*}
F^\vee := \RCHom(F,\CO_X)
\end{equation*}
is called the (derived) {\sf dual object}. 
With the smoothness assumption there is a canonical isomorphism
\begin{equation*}
\RCHom(F,G) \cong F^\vee \lotimes G
\end{equation*}
for all $F,G \in \BD(X)$.


\subsection{Relations}

The functors we introduced so far obey a long list of relations. 
Here we discuss the most important of them.

\subsubsection{Functoriality}

Let $X\xrightarrow{\ f\ } Y \xrightarrow{\ g\ } Z$ be a pair of morphisms. Then
\begin{align*}
R(g\circ f)_* & \cong Rg_*\circ Rf_*,\\
L(g\circ f)^* & \cong Lf^* \circ Lg^*,\\
(g\circ f)^! & \cong f^! \circ g^!.
\end{align*}

In particular, if $Z = \Spec\kk$ is the point then the first formula gives an isomorphism $R\Gamma\circ Rf_* \cong R\Gamma$.

\subsubsection{Local adjunctions}

There are isomorphisms
\begin{align*}
Rf_*\RCHom(Lf^*(F),G) & \cong \RCHom(F,Rf_*(G)),\\
Rf_*\RCHom(G,f^!(F))  & \cong \RCHom(Rf_*(G),F).
\end{align*}

If one applies the functor $R\Gamma$ to these formuals, the usual adjunctions are recovered.
Another local adjunction is the following isomorphism
\begin{equation*}
\RCHom(F\lotimes G,H) \cong \RCHom(F,\RCHom(G,H)).
\end{equation*}

\subsubsection{Tensor products and pullbacks}

Derived tensor product is associative and commutative, thus $\BD(X)$ is a tensor (symmetric monoidal) category.
The pullback functor is a tensor functor, i.e.\ 
\begin{align*}
Lf^*(F\lotimes G) & \cong Lf^*(F) \lotimes Lf^*(G),\\
Lf^*\RCHom(F,G) & \cong \RCHom(Lf^*(F),Lf^*(G)).
\end{align*}

\subsubsection{The projection formula}

In a contrast with the pullback, the pushforward is not a tensor functor. It has, however, a weaker property
\begin{equation*}
Rf_*(Lf^*F \lotimes G) \cong F \lotimes Rf_*(G),
\end{equation*}
which is called {\sf projection formula} and is very useful. A particular case of it is the following isomorphism
\begin{equation*}
Rf_*(Lf^*F) \cong F \lotimes Rf_*(\CO_X).
\end{equation*}

\subsubsection{Base change}

Let $f:X \to S$ and $u:T \to S$ be morphisms of schemes. Consider the fiber product $X_T := X\times_S T$ and the fiber square
\begin{equation}\label{bcsq}
\vcenter{\xymatrix{
X_T \ar[r]^{u_X} \ar[d]_{f_T} & X \ar[d]^f \\
T \ar[r]^u & S
}}
\end{equation}
Using adjunctions and functoriality of pullbacks and pushforwards, it is easy to construct a canonical morphism of functors $Lu^*\circ Rf_* \to Rf_{T*}\circ Lu_X^*$.
The base change theorem says that it is an isomorphism under appropriate conditions. To formulate these we need the following


\begin{definition}
A pair of morphisms $f:X \to S$ and $u:T \to S$ is called {\sf $\Tor$-independent} if for all points 
$x \in X$, $t\in T$ such that $f(x) = s = u(t)$ one has
\begin{equation*}
\Tor_i^{\CO_{S,s}}(\CO_{X,x},\CO_{T,t}) = 0
\qquad\text{for $i > 0$}.
\end{equation*}
\end{definition}

\begin{remark}
If either $f$ or $u$ is flat then the square is $\Tor$-independent. Furthermore, 
when $X$, $S$, and $T$ are all smooth there is a simple sufficient condition for a pair $(f,u)$ to be $\Tor$-independent:
\begin{equation*}
\dim X_T = \dim X + \dim T - \dim S,
\end{equation*}
i.e.\ the equality of the dimension of $X_T$ and of its expected dimension.
\end{remark}


\begin{theorem}[\cite{K06a}]
The base change morphism $Lu^*\circ Rf_* \to Rf_{T*}\circ Lu_X^*$ is an isomorphism
if and only if the pair of morphisms $f:X \to S$ and $u:T \to S$ is $\Tor$-independent.
\end{theorem}

%

\subsection{Fourier--Mukai functors}

Let $X$ and $Y$ be algebraic varieties and $K \in \BD(X\times Y)$ an object
of the derived category of the product. Given this data we define a functor
\begin{equation*}
\Phi_K:\BD(X) \to \BD(Y),
\qquad
F \mapsto Rp_{Y*}(K \lotimes Lp_X^*(F)),
\end{equation*}
where $p_X:X\times Y \to  X$ and $p_Y:X \times Y \to Y$ are the projections.
It is called the {\sf Fourier--Mukai functor} (another names are 
the {\sf kernel functor}, the {\sf integral functor}) with kernel $K$.

Fourier--Mukai functors form a nice class of functors, which includes most of the functors
we considered before. This class is closed under compositions and adjunctions.

\begin{exercise}\label{fmpf}
Let $f:X \to Y$ be a morphism. Let $\gamma_f:X \to X\times Y$ be the graph of $f$. 
Show that the Fourier--Mukai functor with kernel $\gamma_{f*}\CO_X \in \BD(X\times Y)$ is isomorphic 
to the derived pushforward $Rf_*$.
\end{exercise}

\begin{exercise}\label{fmpb}
Let $g:Y \to X$ be a morphism. Let $\gamma_g:Y \to X\times Y$ be the graph of $g$. 
Show that the Fourier--Mukai functor with kernel $\gamma_{g*}\CO_Y \in \BD(X\times Y)$ is isomorphic 
to the derived pullback $Lg^*$.
\end{exercise}

\begin{exercise}\label{fmtp}
Let $\CE \in \BD(X)$ be an object. Let $\delta:X \to X\times X$ be the diagonal embedding. 
Show that the Fourier--Mukai functor with kernel $R\delta_*\CE \in \BD(X\times X)$ is isomorphic 
to the derived tensor product functor $\CE\lotimes -$.
\end{exercise}

\begin{exercise}
Let $K_{12} \in \BD(X_1\times X_2)$ and $K_{23} \in \BD(X_2\times X_3)$.
Consider the triple product $X_1\times X_2\times X_3$ and the projections $p_{ij}:X_1\times X_2\times X_3 \to X_i\times X_j$.
Show an isomorphism of functors
\begin{equation*}
\Phi_{K_{23}} \circ \Phi_{K_{12}} \cong \Phi_{K_{12}\circ K_{23}},
\end{equation*}
where $K_{12}\circ K_{23}$ is the {\sf convolution of kernels}, defined by
\begin{equation*}
K_{12}\circ K_{23} := Rp_{13*}(Lp_{12}^*K_{12} \lotimes Lp_{23}^*K_{23}).
\end{equation*}
\end{exercise}

Adjoint functors of Fourier--Mukai functors are also Fourier--Mukai functors.


\begin{lemma}
The right adjoint functor of $\Phi_K$ is the Fourier--Mukai functor
\begin{equation*}
\Phi_K^! \cong \Phi_{K^\vee\lotimes\omega_X[\dim X]} : \BD(Y) \to \BD(X).
\end{equation*}
The left adjoint functor of $\Phi_K$ is the Fourier--Mukai functor
\begin{equation*}
\Phi_K^* \cong \Phi_{K^\vee\lotimes\omega_Y[\dim Y]}  : \BD(Y) \to \BD(X).
\end{equation*}
\end{lemma}
\begin{proof}
The functor $\Phi_K$ is the composition of the derived pullback $Lp_X^*$ , the derived tensor product with~$K$, and the derived pushforward $Rp_{Y*}$.
Therefore its right adjoint functor is the composition of their right adjoint functors, i.e.\ 
of the twisted pullback functor $p_Y^!$, the derived tensor product with $K^\vee$, and the derived pushforward functor $Rp_{X*}$.
By Grothendieck duality we have $p_Y^!(G) \cong Lp_Y^*(G) \lotimes \omega_X[\dim X]$. Altogether, this gives the first formula.
%
For the second part note that if $K' = K^\vee\lotimes\omega_Y[\dim Y]$ then the functor $\Phi_{K'}:\BD(Y) \to \BD(X)$
has a right adjoint, which by the first part of the Lemma coincides with~$\Phi_K$. Hence the left adjoint of $\Phi_K$ is $\Phi_{K'}$.
\end{proof}

\subsection{Serre functor}

The notion of a Serre functor is a categorical interpretation of the Serre duality.

\begin{definition}
A {\sf Serre functor} in a triangulated category $\CT$ is an autoequivalence $\SS_\CT:\CT \to \CT$
with a bifunctorial isomorphism
\begin{equation*}
\Hom(F,G)^\vee \cong \Hom(G,\SS_\CT(F)).
\end{equation*}
\end{definition}
It is easy to show (see~\cite{BK}) that if a Serre functor exists, it is unique up to a canonical isomorphism.
When $\CT = \BD(X)$, the Serre functor is given by a simple formula
\begin{equation*}
\SS_{\BD(X)}(F) = F \otimes \omega_X[\dim X].
\end{equation*}
The bifunctorial isomorphism in its definition is the Serre duality for $X$.

If $X$ is a Calabi--Yau variety (i.e.\ $\omega_X \cong \CO_X$) then the corresponding Serre functor $\SS_{\BD(X)} \cong [\dim X]$ is just a shift.
This motivates the following 

\begin{definition}[\cite{K15}]\label{definition-cy-cat}
A triangulated category $\CT$ is a {\sf Calabi--Yau category of dimension $n \in \ZZ$} if $\SS_\CT \cong [n]$.
A~triangulated category $\CT$ is a {\sf fractional Calabi--Yau category of dimension $p/q \in \QQ$} if $\SS_\CT^q \cong [p]$.
\end{definition}

Of course, $\BD(X)$ cannot be a fractional Calabi--Yau category with a non-integer Calabi--Yau dimension.
However, we will see soon some natural and geometrically meaningful fractional Calabi--Yau categories. 

\subsection{Hochschild homology and cohomology}

Hochschild homology and cohomology of algebras are well-known and important invariants.
One can define them for triangulated categories as well (under some technical assumptions).
A non-rigorous definition is the following
\begin{equation*}
\HOH^\bullet(\CT) = \Ext^\bullet(\id_\CT,\id_\CT),
\qquad
\HOH_\bullet(\CT) = \Ext^\bullet(\id_\CT,\SS_\CT).
\end{equation*}
Thus Hochschild cohomology $\HOH^\bullet(\CT)$ is the self-$\Ext$-algebra of the identity functor of $\CT$, and 
Hochschild homology $\HOH_\bullet(\CT)$ is the $\Ext$-space from the identity functor to the Serre functor. 
For a more rigorous definition one should choose a DG-enhancement of the triangulated
category $\CT$, replace the identity and the Serre functor by appropriate bimodules,
and compute $\Ext$-spaces in the derived category of DG bimodules.
In a geometrical situation (i.e., when $\CT = \BD(X)$ with $X$ smooth and projective)
one can replace the identity functor by the structure sheaf of the diagonal and the Serre
functor by the diagonal pushforward of $\omega_X[\dim X]$ (i.e.~Fourier--Mukai kernel
of the Serre functor) and compute $\Ext$-spaces in the derived category of coherent 
sheaves on the square $X \times X$ of the variety.

When $\CT$ is geometric, Hochschild homology and cohomology have an interpretation in terms 
of standard geometrical invariants.

\begin{theorem}[Hochschild--Kostant--Rosenberg]\label{theorem-hkr}
If $\CT = \BD(X)$ with $X$ smooth and projective, then
\begin{equation*}
\HOH^k(\CT) = \bigoplus_{q+p = k} H^q(X,\Lambda^pT_X),
\qquad
\HOH_k(\CT) = \bigoplus_{q-p = k} H^q(X,\Omega^p_X) = \bigoplus_{q-p = k} H^{p,q}(X).
\end{equation*}
\end{theorem}

Thus the Hochschild cohomology is the cohomology of polyvector fields, while the Hochschild homology
is the cohomology of differential forms, or equivalently, the Hodge cohomology of $X$
with one grading lost.

\begin{example}\label{example-hoh-dim-0-1}
If $\CT = \BD(\Spec(\kk))$ then $\HOH^\bullet(\CT) = \HOH_\bullet(\CT) = \kk$.
If $\CT = \BD(C)$ with $C$ being a smooth projective curve of genus $g$, then
\begin{equation*}
\HOH^\bullet(\CT) = 
\begin{cases}
\kk \oplus \fsl_2[-1], & \text{if $g = 0$},\\
\kk \oplus \kk^2[-1] \oplus \kk[-2], & \text{if $g = 1$},\\
\kk \oplus \kk^g[-1] \oplus \kk^{3g-3}[-2], & \text{if $g \ge 2$},
\end{cases}
\qquad\text{and}\qquad
\HOH_\bullet(\CT) = \kk^g[1] \oplus \kk^2 \oplus \kk^g[-1].
\end{equation*}
\end{example}

Hochschild cohomology is a graded algebra (in fact, it is what is called
a Gerstenhaber algebra, having both an associative multiplication and a Lie bracket of degree $-1$), 
and Hochschild homology is a graded module over it. When the category $\CT$ is a Calabi--Yau category, it is a free module.

\begin{lemma}\label{lemma-hoh-cy}
If $\CT$ is a Calabi--Yau category of dimension $n \in \ZZ$ then 
\begin{equation*}
\HOH_\bullet(\CT) \cong \HOH^\bullet(\CT)[n].
\end{equation*}
\end{lemma}
\begin{proof}
By definition of a Calabi--Yau category we have $\SS_\CT = \id_\CT[n]$. Substituting this into
the definition of Hochschild homology we get the result. 
\end{proof}

For a more accurate proof one should do the same with DG bimodules or with sheaves on $X\times X$, see~\cite{K15} for details.

\section{Semiorthogonal decompositions}

A semiorthogonal decomposition is a way to split a triangulated category into smaller pieces.

\subsection{Two-step decompositions}

We start with the following simplified notion.

\begin{definition}
A (two-step) {\sf semiorthogonal decomposition} ({\sf s.o.d.}\ for short) of a triangulated category~$\CT$ is a pair
of full triangulated subcategories $\CA,\CB \subset \CT$ such that
\begin{itemize}
\item $\Hom(B,A) = 0$ for any $A \in \CA$, $B \in \CB$;
\item for any $T \in \CT$ there is a distinguished triangle
\begin{equation}\label{equation-sod-triangle}
T_\CB \to T \to T_\CA \to T_\CB[1]
\end{equation}
with $T_\CA \in \CA$ and $T_\CB \in \CB$.
\end{itemize}
\end{definition}

An s.o.d.\ is denoted by $\CT = \langle \CA, \CB \rangle$.
Before giving an example let us make some observations.

\begin{lemma}
If $\CT = \langle \CA, \CB \rangle$ is an s.o.d, then for any $T \in \CT$ the triangle~\eqref{equation-sod-triangle}
is unique and functorial in $T$. In partcular, the association $T \mapsto T_\CA$ is a functor $\CT \to \CA$, 
left adjoint to the embeding functor $\CA \to \CT$, and the association $T \mapsto T_\CB$ is a functor $\CT \to \CB$, 
right adjoint to the embeding functor $\CB \to \CT$.
\end{lemma}
\begin{proof}
Let $T,T' \in \CT$ and $\varphi \in \Hom(T,T')$. Let~\eqref{equation-sod-triangle} and
\begin{equation*}
T'_\CB \to T' \to T'_\CA \to T'_\CB[1]
\end{equation*}
be the corresponding decomposition triangles. Consider the long exact sequence obtained by applying the functor $\Hom(-,T'_\CA)$ to~\eqref{equation-sod-triangle}:
\begin{equation*}
\dots \to \Hom(T_\CB[1],T'_\CA) \to \Hom(T_\CA,T'_\CA) \to \Hom(T,T'_\CA) \to \Hom(T_\CB,T'_\CA) \to \dots
\end{equation*}
By semiorthogonality the left and the right terms are zero. Hence $\Hom(T_\CA,T'_\CA) \cong \Hom(T,T'_\CA)$.
This means that the composition $T \xrightarrow{\ \varphi\ } T' \to T'_\CA$ factors in a unique way as a composition
$T \to T_\CA \to T'_\CA$. Denoting the obtained morphism $T_\CA \to T'_\CA$ by $\varphi_\CA$, the uniqueness implies
the functoriality of $T \mapsto T_\CA$. Furthermore, taking an arbitrary object $A \in \CA$ and applying the functor 
$\Hom(-,A)$ to~\eqref{equation-sod-triangle} and using again the semiorthogonality, we deduce an isomorphism
$\Hom(T_\CA,A) \cong \Hom(T,A)$, which means that the functor $T \mapsto T_\CA$ is left adjoint to the embedding $\CA \to \CT$.
The functoriality of $T \mapsto T_\CB$ and its adjunction are proved analogously.
\end{proof}

Note also that the composition of the embedding $\CA \to \CT$ with the projection $\CT \to \CA$ is the identity.

In fact, the above construction can be reversed.

\begin{lemma}\label{lemma-admissible-sod}
Assume $\alpha:\CA \to \CT$ is a triangulated functor, which has a left adjoint $\alpha^*:\CT \to \CA$ and $\alpha^*\circ\alpha \cong \id_\CA$
{\rm(}such $\CA$ is called {\sf left admissible}{\rm)}. Then $\alpha$ is full and faithful and there is an s.o.d.\ 
\begin{equation*}
\CT = \langle \alpha(\CA), \Ker\alpha^* \rangle. 
\end{equation*}
Analogously, if $\beta:\CB \to \CT$ is a triangulated functor, which has a right adjoint $\beta^!:\CT \to \CB$ and $\beta^!\circ\beta \cong \id_\CB$
{\rm(}such $\CB$ is called {\sf right admissible}{\rm)}, then $\beta$ is full and faithful and there is an s.o.d.\ 
\begin{equation*}
\CT = \langle \Ker\beta^!, \beta(\CB) \rangle.
\end{equation*}
\end{lemma}
\begin{proof}
If $B \in \Ker\alpha^*$ then by adjunction $\Hom(B,\alpha(A)) = \Hom(\alpha^*(B),A) = 0$. Further, for any $T \in \CT$ consider the 
unit of adjunction $T \to \alpha\alpha^*(T)$ and extend it to a distinguished triangle
\begin{equation*}
T' \to T \to \alpha\alpha^*(T) \to T'[1].
\end{equation*}
Applying $\alpha^*$ we get a distinguished triangle
\begin{equation*}
\alpha^*T' \to \alpha^*T \to \alpha^*\alpha\alpha^*(T) \to \alpha^*T'[1].
\end{equation*}
If we show that the middle map is an isomorphism, it would follow that $\alpha^*T' = 0$, hence $T' \in \Ker\alpha^*$.
For this consider the composition $\alpha^*T \to \alpha^*\alpha\alpha^*(T) \to \alpha^*T$, where the first map is the morphism
from the triangle (it is induced by the unit of the adjunction), and the second map is induced by the counit of the adjunction.
The composition of these maps is an isomorphism
by one of the definitions of adjunction. Moreover, the second morphism is an isomorphism by the condition of the Lemma. 
Hence the first morphism is also an isomorphism. As we noted above, this implies that $T' \in \Ker\alpha^*$, and so 
the above triangle is a decomposition triangle for $T$. This proves the first semiorthogonal decomposition.
The second statement is proved analogously.
\end{proof}


\begin{example}\label{example-sod-ox}
Let $X$ be a $\kk$-scheme with the structure morphism $\pi_X:X \to \Spec(\kk)$. If $H^\bullet(X,\CO_X) = \kk$ 
then there is a semiorthogonal decomposition
\begin{equation*}
\BD(X) = \langle \Ker R\pi_{X*}, L\pi_X^*(\BD(\Spec(\kk)) \rangle.
\end{equation*}
Indeed, the functor $R\pi_{X*}$ is right adjoint to $L\pi_X^*$ and by projection formula
\begin{equation*}
R\pi_{X*}(L\pi_X^*(F)) \cong F \lotimes R\pi_{X*}(\CO_X) = F \lotimes H^\bullet(X,\CO_X) = F \lotimes \kk = F.
\end{equation*}
Therefore, the functor $L\pi_X^*$ is fully faithful and right admissible, so the second semiorthogonal decomposition of Lemma~\ref{lemma-admissible-sod} applies.
\end{example}

\begin{exercise}\label{exercise-sod-ox-2}
Show that there is also a semiorthogonal decomposition
\begin{equation*}
\BD(X) = \langle  L\pi_X^*(\BD(\Spec(\kk)), \Ker (R\pi_{X*}\circ\SS_X) \rangle.
\end{equation*}
\end{exercise}

Note that $\BD(\Spec(\kk))$ is the derived category of $\kk$-vector spaces (so we will denote it simply by~$\BD(\kk)$), 
its objects can be thought of just as graded vector spaces, and the functor $L\pi_X^*$ applied to 
a graded vector space $V^\bullet$ is just $V^\bullet\otimes\CO_X$ (a complex of trivial vector bundles with zero differentials).
Moreover, $R\pi_{X*}(-) \cong \Ext^\bullet(\CO_X,-)$, so its kernel is the orthogonal subcategory
\begin{equation*}
\CO_X^\perp = \{ F \in \BD(X) \mid \Ext^\bullet(\CO_X,F) = 0 \}.
\end{equation*}
The decompositions of Example~\ref{example-sod-ox} and Exercise~\ref{exercise-sod-ox-2} thus can be rewritten as
\begin{equation*}
\BD(X) = \langle \CO_X^\perp, \CO_X \rangle 
\quad\text{and}\quad
\BD(X) = \langle \CO_X, {}^\perp\CO_X \rangle, 
\end{equation*}
where we write just $\CO_X$ instead of $\CO_X \otimes \BD(\kk)$.

\begin{example}\label{example-sod-e}
Let $E$ be an object in $\BD(X)$ such that $\Ext^\bullet(E,E) = \kk$ (so-called {\sf exceptional object}).
Then there are semiorthogonal decompositions
\begin{equation*}
\BD(X) = \langle E^\perp, E \rangle
\quad\text{and}\quad
\BD(X) = \langle E, {}^\perp E \rangle, 
\end{equation*}
where again we write $E$ instead of $E \otimes \BD(\kk)$.
\end{example}

\subsection{General semiorthogonal decompositions} 

The construction of Lemma~\ref{lemma-admissible-sod} can be iterated to produce a longer (multi-step) semiorthogonal decomposition.

\begin{definition}
A {\sf semiorthogonal decomposition} of a triangulated category $\CT$ is a sequence $\CA_1,\dots,\CA_m$
of its full triangulated subcategories, such that
\begin{enumerate}
\item $\Hom(\CA_i,\CA_j) = 0$ for $i > j$, and
\item for any $T \in \CT$ there is a chain of morphisms $0 = T_m \to T_{m-1} \to \dots \to T_1 \to T_0 = T$
such that $\Cone(T_i \to T_{i-1}) \in \CA_i$ for each $1 \le i \le m$.
\end{enumerate}
The notation for a semiorthogonal decomposition is $\CT = \langle \CA_1,\CA_2,\dots,\CA_m \rangle$.
\end{definition}

\begin{lemma}\label{lemma-sod-from-ec}
Assume $E_1,\dots,E_n$ is a sequence of exceptional objects such that $\Ext^\bullet(E_i,E_j) = 0$ for $i > j$
\rm{(}this is called an {\sf exceptional collection}{\rm)}. Then
\begin{equation*}
\BD(X) = \langle E_1^\perp \cap \dots \cap E_n^\perp, E_1, \dots, E_n \rangle
\end{equation*}
is a semiorthogonal decomposition.
\end{lemma}
\begin{proof}
Write the s.o.d.\ of Example~\ref{example-sod-e} for $E_n \in \BD(X)$. Then note that $E_1, \dots, E_{n-1} \in E_n^\perp$.
Then write analogous s.o.d for the exceptional object $E_{n-1}$ in the triangulated category $E_n^\perp$ and repeat until the required s.o.d.\ is obtained.
\end{proof}

\begin{remark}
If the first component $E_1^\perp \cap \dots \cap E_n^\perp$ of the above decomposition is zero, 
one says that $E_1,\dots,E_n$ is a {\sf full exceptional collection}.
\end{remark}

\begin{remark}\label{remark-mutations}
Whenever a semiorthogonal decomposition $\CT = \langle \CA_1, \dots, \CA_m \rangle$ is given, 
one can construct many other semiorthogonal decompositions by operations called ``mutations''.
Roughly speaking, to mutate a semiorthogonal decomposition we omit one of its components $\CA_i$
and then insert a new component in a different position. The new component is abstractly equivalent
to $\CA_i$, but is embedded differently into $\CT$. Mutations induce an action of the braid group
on the set of all semiorthogonal decompositions of a given triangulated category.
\end{remark}

\subsection{Semiorthogonal decompositions for Fano varieties}

Most interesting examples of semiorthogonal decompositions come from Fano varieties. Recall that a {\sf Fano variety} is 
a smooth projective connected variety~$X$ with ample anticanonical class $-K_X$.
A Fano variety $X$ is {\sf prime} if $\Pic(X) \cong \ZZ$. The {\sf index} of a prime Fano variety $X$ 
is the maximal integer $r$ such that $-K_X = rH$ for some $H \in \Pic(X)$. 

\begin{example}\label{example-fano-index-r}
Let $X$ be a Fano variety of index $r$ with $-K_X = rH$. Then the collection of line bundles 
$(\CO_X((1-r)H), \dots, \CO_X(-H),\CO_X)$ is an exceptional collection.
Indeed, for $i > -r$ we have 
\begin{equation*}
H^{>0}(X,\CO_X(iH)) = H^{>0}(X,K_X((i+r)H)) = 0
\end{equation*}
by Kodaira vanishing. Moreover, $H^0(X,\CO_X(iH)) = 0$ for $i < 0$ by ampleness of $H$ and $H^0(X,\CO_X) = \kk$ by connectedness of $X$.
Applying Lemma~\ref{lemma-sod-from-ec} we obtain a semiorthogonal decomposition
\begin{equation*}
\BD(X) = \langle \CA_X, \CO_X((1-r)H), \dots, \CO_X(-H), \CO_X \rangle.
\end{equation*}
with $\CA_X$ being the orthogonal complement of the sequence.
\end{example}

In some cases, the orthogonal complement $\CA_X$ appearing in the above semiorthogonal decomposition vanishes or can be explicitly described.

\begin{example}\label{example-pn}
The projective space $\PP^n$ is a Fano variety of index $r = n + 1$. The orthogonal complement of the maximal exceptional
sequence of line bundles vanishes and we have a full exceptional collection
\begin{equation*}
\BD(\PP^n) = \langle \CO(-n),\dots,\CO(-1),\CO \rangle. 
\end{equation*}
\end{example}

\begin{example}\label{example-qn}
A smooth quadric $Q^n \subset \PP^{n+1}$ is a Fano variety of index $r = n$. The orthogonal complement of the maximal exceptional
sequence of line bundles does not vanish but can be explicitly described. In fact, there is a semiorthogonal decomposition
\begin{equation*}
\BD(Q^n) = \langle \BD(\Cl_0),\CO(1-n),\dots,\CO(-1),\CO \rangle, 
\end{equation*}
where $\Cl_0$ is the even part of the corresponding Clifford algebra.
If the field $\kk$ is algebraically closed of characteristic $0$ then $\Cl_0$ is Morita equivalent to $\kk$, if $n$ is odd, or 
to $\kk \times \kk$, if $n$ is even. So in this case the category $\BD(\Cl_0)$ is generated by one or two (completely orthogonal) 
exceptional objects, which in terms of $\BD(Q^n)$ are given by the spinor bundles.
\end{example}

In fact, the semiorthogonal decomposition of Example~\ref{example-qn} is much more general.
It is valid for arbitrary base fields (of odd characteristic) and also for non-smooth quadrics.
There is also an analog for flat families of quadrics over nontrivial base schemes (see~\cite{K08}).

If $X$ is a Fano variety of dimension 2 (i.e.\ a del Pezzo surface) and the base field is $\C$,
then $\BD(X)$ has a full exceptional collection. This follows easily from representation of $X$
as a blow up of $\PP^2$ in several points and Orlov's blowup formula (see Theorem~\ref{theorem-sod-blowup} below).
For more general fields the situation is more complicated.


\subsection{Fano 3-folds}\label{ssection-fano3}

The situation becomes much more interesting when one goes into dimension 3. Fano threefolds were completely
classified in works of Fano, Iskovskih, Mori and Mukai. There are 105 deformation families of those (quite a large number!),
so we restrict our attention to prime Fano 3-folds. These form only 17 families.

The index of Fano 3-folds is bounded by $4$. The only Fano 3-fold of index $4$ is the projective space~$\PP^3$,
and the only Fano 3-fold of index $3$ is the quadric $Q^3$, their derived categories are described in Examples~\ref{example-pn} and~\ref{example-qn}.
The structure of derived categories of prime Fano 3-folds of index 1 and 2 is summarized in the following table
(see \cite{O91,BO95,K96,K04,K05,K06a,K09a,K14,K15} for details).


\begin{table}[h]\label{table-fano}
\caption{Derived categories of prime Fano 3-folds}
\begin{tabular}{|l|l|l|l|l|}
\hline
\raisebox{-2ex}{\rule{0pt}{4.5ex}}
Index 1 & Index 2 & Relation & Serre functor & $\cydim$ \\
\hline
\raisebox{-2ex}{\rule{0pt}{4.5ex}}
$\BD(X_{22}) = \langle E_4,E_3,E_2,\CO \rangle$ & 
$\BD(X_{5}) = \langle E_2(-H),\CO(-H),E_2,\CO \rangle$ 
&&& \\ \hline
\raisebox{-2ex}{\rule{0pt}{4.5ex}}
$\BD(X_{18}) = \langle \BD(C_2),E_2,\CO \rangle$ & 
$\BD(V_{2,2}) = \langle \BD(C_2),\CO(-H),\CO \rangle$ 
&&& \\ \hline
\raisebox{-2ex}{\rule{0pt}{4.5ex}}
$\BD(X_{16}) = \langle \BD(C_3),E_3,\CO \rangle$ & 
&&& \\ \hline
\raisebox{-2ex}{\rule{0pt}{4.5ex}}
$\BD(X_{14}) = \langle \CA_{X_{14}},E_2,\CO \rangle$ & 
$\BD(V_{3}) = \langle \CA_{V_3},\CO(-H),\CO \rangle$ &
$\CA_{X_{14}} \cong \CA_{V_3}$ & $\SS^3 \cong [5]$ & $1\frac23$ 
\\ \hline
\raisebox{-2ex}{\rule{0pt}{4.5ex}}
$\BD(X_{12}) = \langle \BD(C_7),E_5,\CO \rangle$ & 
&&& \\ \hline
\raisebox{-2ex}{\rule{0pt}{4.5ex}}
$\BD(X_{10}) = \langle \CA_{X_{10}},E_2,\CO \rangle$ & 
$\BD(dS_{4}) = \langle \CA_{dS_4},\CO(-H),\CO \rangle$ &
$\CA_{X_{10}} \sim \CA_{dS_4}$ & $\SS^2 \cong [4]$ & $2$ \\ \hline
\raisebox{-2ex}{\rule{0pt}{4.5ex}}
$\BD(V_{2,2,2}) = \langle \CA_{V_{2,2,2}},\CO \rangle$ & 
&&& \\ \hline
\raisebox{-2ex}{\rule{0pt}{4.5ex}}
$\BD(V_{2,3}) = \langle \CA_{V_{2,3}},E_2,\CO \rangle$ & 
$\BD(V_{6}^{1,1,1,2,3}) = \langle \CA_{V_{6}},\CO(-H),\CO \rangle$ &
$\CA_{V_{2,3}} \sim \CA_{V_{6}}$ & $\SS^3 \cong [7]$ & $2\frac13$ \\ \hline
\raisebox{-2ex}{\rule{0pt}{4.5ex}}
$\BD(V_{4}) = \langle \CA_{V_{4}},\CO \rangle$ & 
&& $\SS_{\CA_{V_4}}^4 \cong [10]$ & $2\frac12$ \\ \hline
\raisebox{-2ex}{\rule{0pt}{4.5ex}}
$\BD(dS_{6}) = \langle \CA_{dS_{6}},\CO \rangle$ & 
&& $\SS_{\CA_{dS_6}}^6 \cong [16]$ & $2\frac23$ \\ \hline
\end{tabular}
\end{table}

Here $V_{d_1,\dots,d_r}^{w_0,\dots,w_{r+3}}$ denotes a smooth complete intersection of multidegree $(d_1,\dots,d_r)$ 
in a weighted projective space with weights $(w_0,\dots,w_{r+3})$ (and if all weights are equal to 1 we omit them).
The notation $X_d$ means a smooth Fano 3-fold of degree $d$. Further, $dS_d$ stands for the degree $d$ double solid,
i.e., the double covering of $\PP^3$ branched in a smooth divisor of degree $d$. 

\bigskip 

As we discussed in Example~\ref{example-fano-index-r} the structure sheaf $\CO_X$ on a Fano 3-fold $X$ is always exceptional,
and in case of index 2 (the second column of the table) it extends to an exceptional pair $(\CO_X(-H),\CO_X)$ (where as usual
$H$ denotes the ample generator of the Picard group). Whenever there is an additional exceptional vector bundle of rank $r$,
we denote it by $E_r$ (see one of the above references for the construction of these). For Fano 3-folds $X_{22}$ and $X_5$
(the first line of the table) one can construct in this way a full exceptional collection.

For other prime Fano 3-folds there is no full exceptional collection (this follows from Corollary~\ref{corollary-fec-hoh} below),
so the derived category contains an additional component. In four cases (varieties $X_{18}$, $X_{16}$, $X_{12}$, and $V_{2,2}$)
this component can be identified with the derived category of a curve (of genus $2$, $3$, and $7$). For the other Fano 3-folds,
the extra component $\CA_X$ cannot be described in a simple way. However, it has some interesting properties. For example,
in most cases it is a fractional Calabi--Yau category (as defined in Definition~\ref{definition-cy-cat}, see~\cite{K15} for the proofs). 
We list in the fourth column of the table the CY property of the Serre functor of these categories, and in the fifth column their CY dimension.

\bigskip

It is a funny and mysterious observation that the nontrivial components of Fano 3-folds sitting in the same row of the table
have very much in common. In fact they are equivalent for rows 1, 2, and 4, and are expected to be deformation equivalent
for rows 6 and 8 (see~\cite{K09a} for precise results and conjectures). We mention this relation in the third column
of the table, by using the sign ``$\cong$'' for equivalence and the sign ``$\sim$'' for deformation equivalence 
of categories.

\bigskip


It is interesting to compare this table with the table in Section 2.3 of~\cite{Bea15}, where the known information
about birational properties of prime Fano 3-folds is collected (the notation we use agrees with the notation in {\em loc.\ cit.}).
From the comparison it is easy to see that the structure of the derived categories of Fano 3-folds correlates
with their rationality properties. All Fano 3-folds which are known to be rational ($X_{22}$, $X_{18}$, $X_{16}$, $X_{12}$, $X_5$, and $V_{2,2}$)
have semiorthogonal decompositions consisting only of exceptional objects (derived categories of points)
and derived categories of curves. On a contrary, those Fano 3-folds which are known or expected to be nonrational,
have some ``nontrivial pieces'' in their derived categories. 
Note also that typically these ``nontrivial pieces'' are fractional Calabi--Yau categories.

In the next section we will try to develop this observation into a birational invariant.


\begin{remark}
Among the most interesting examples of nonrational varieties are the Artin--Mumofrd double solids (see~\cite{AM}).
If $X$ is one of these double solids and $\TX$ is its blowup then there is a semiorthogonal decomposition
\begin{equation*}
\BD(\TX) = \langle \CA_\TX, \CO_\TX(-H), \CO_\TX \rangle.
\end{equation*}
The category $\CA_\TX$ is equivalent to the derived category of the associated Enriques surface (see~\cite{IK,HT}). 
This explains the torsion in the cohomology of $X$ which implies its non-rationality.
\end{remark}

\subsection{Hochschild homology and semiorthogonal decompositions}

A fundamental property of Hoch\-schild homology, which makes it very useful, is its additivity
with respect to semiorthogonal decompositions.

\begin{theorem}[\cite{K09b}]\label{theorem-hoh-additivity}
If $\CT = \langle \CA_1,\dots,\CA_m \rangle$ is a semiorthogonal decomposition, then
\begin{equation*}
\HOH_\bullet(\CT) = \bigoplus_{i=1}^m \HOH_\bullet(\CA_i).
\end{equation*}
\end{theorem}

One of the nice consequences of this result is the following necessary condition for a category 
to have a full exceptional collection.

\begin{corollary}\label{corollary-fec-hoh}
If a category $\CT$ has a full exceptional collection then $\HOH_k(\CT) = 0$ for $k \ne 0$
and $\dim \HOH_0(\CT) < \infty$. Moreover, the length of any full exceptional collection
in $\CT$ equals $\dim \HOH_0(\CT)$.

In particular, if $X$ is a smooth projective variety and $\BD(X)$ has a full exceptional
collection, then $H^{p,q}(X) = 0$ for $p \ne q$, and the length of the exceptional collection
equals to $\sum \dim H^{p,p}(X)$.
\end{corollary}
\begin{proof}
The first part follows from the additivity Theorem and Example~\ref{example-hoh-dim-0-1}.
The second part follows from the first and the Hochschild--Kostant--Rosenberg isomorphism of Theorem~\ref{theorem-hkr}.
\end{proof}

\begin{remark}
Note that the conditions of the Corollary are only necessary. The simplest example
showing they are no sufficient is provided by the derived category of an Enriques surface $S$.
It is well known that the Hodge numbers of $S$ are zero away from the diagonal of the Hodge diamond, 
however a full exceptional collection in $\BD(S)$ cannot exist since the Grothendieck group $K_0(\BD(S))$
has torsion.
\end{remark}

In a contrast, the Hochschild cohomology is not additive (it depends not only on the components 
of a semiorthogonal decomposition, but also on the way they are glued together). However, if there is a completely
orthogonal decomposition, then the Hochschild cohomology is additive.

\begin{lemma}[\cite{K09b}]
If $\CT = \langle \CA, \CB \rangle$ is a completely orthogonal decomposition then 
\begin{equation*}
\HOH^\bullet(\CT) = \HOH^\bullet(\CA) \oplus \HOH^\bullet(\CB).
\end{equation*}
\end{lemma}

This result also has a nice consequence. Recall that for $\CT = \BD(X)$ we have $\HOH^0(\CT) = H^0(X,\CO_X)$
by the HKR isomorphism. This motivates the following

\begin{definition}
A triangulated category $\CT$ is called {\sf connected}, if $\HOH^0(\CT) = \kk$.
\end{definition}

\begin{corollary}\label{corollary-orthogonal-indecomposability}
If $\CT$ is a connected triangulated category then $\CT$ has no completely orthogonal decompositions.
\end{corollary}
\begin{proof}
If $\CT = \langle \CA,\CB \rangle$ is a completely orthogonal decomposition then $\kk = \HOH^0(\CT) = \HOH^0(\CA) \oplus \HOH^0(\CB)$,
hence one of the summands vanishes. But a nontrivial category always has nontrivial zero Hochschild cohomology (since the identity
functor always has the identity endomorphism).
\end{proof}

\subsection{Indecomposability}

There are triangulated categories which have no nontrivial semiorthogonal decompositions.
First examples were found by Bridgeland in~\cite{Br}.

\begin{proposition}[\cite{K15}]\label{proposition-cy-indecomposable}
If $\CT$ is a connected Calabi--Yau category then $\CT$ has no semiorthogonal decompositions.
\end{proposition}
\begin{proof}
Assume $\CT$ is a Calabi--Yau category of dimension $n$ and $\CT = \langle \CA,\CB \rangle$ is a semiorthogonal decomposition. 
For any $A \in \CA$ and $B \in \CB$ we have
\begin{equation*}
\Hom(A,B)^\vee = \Hom(B,\SS_\CT(A)) = \Hom(B,A[n]) = 0
\end{equation*}
since $A[n] \in \CA$. Hence the decomposition is completely orthogonal and Corollary~\ref{corollary-orthogonal-indecomposability} applies.
\end{proof}

\begin{corollary}\label{corollary-k-indecomposable}
The category $\BD(\kk)$ is indecomposable.
\end{corollary}
\begin{proof}
Indeed, $\SS_{\BD(\kk)} = \id$, so it is Calabi--Yau of dimension $0$.
\end{proof}

Besides, one can check that derived categories of curves of positive genus are indecomposable.

\begin{proposition}[\cite{Ok}]\label{proposition-curve-indecomposable}
Let $\CT = \BD(C)$ with $C$ a smooth projective curve. If $g(C) > 0$ then $\CT$ is indecomposable.
If $C = \PP^1$ then any semiorthogonal decomposition of $\CT$ is given by an exceptional pair.
\end{proposition}

One can also show that many surfaces have no semiorthogonal decompositions, see \cite{KaOk}.

The Bridgeland's result shows that Calabi--Yau categories can be thought of as the simplest building blocks
of other categories. Because of that the following observation is useful.

\begin{lemma}[\cite{K15}]
Assume $\BD(X) = \langle \CA,\CB \rangle$ is a semiorthogonal decomposition with $\CA$ being Calabi--Yau category of dimension $n$.
Then $\dim X \ge n$.
\end{lemma}
\begin{proof}
By Calabi--Yau property of $\CA$ we have $\HOH_{-n}(\CA) = \HOH^0(\CA) \ne 0$ (see Lemma~\ref{lemma-hoh-cy}). 
By additivity of Hochschild homology we then have $\HOH_{-n}(\BD(X)) \ne 0$ as well.
But by HKR isomorphism if $n > \dim X$ then $\HOH_{-n}(\BD(X)) = 0$.
\end{proof}

In fact, most probably in the boundary case $\dim X = n$ there are also strong restrictions.

\begin{conjecture}[\cite{K15}]\label{conjecture-cydim}
Assume $\BD(X) = \langle \CA,\CB \rangle$ is a semiorthogonal decomposition with $\CA$ being Calabi--Yau category of dimension $n = \dim X$.
Then $X$ is a blowup of a Calabi--Yau variety $Y$ of dimension $n$ and $\CA \cong \BD(Y)$.
\end{conjecture}


We will see in the next section that the converse is true.

Unfortunately, analogs of these results for fractional Calabi--Yau categories are not known,
and in fact, some of them are just not true. For example, if $X$ is a smooth cubic surface in $\PP^3$
and $\CA_X = \CO_X^\perp \subset \BD(X)$, then $\CA_X$ is a fractional Calabi--Yau category
of dimension $4/3$, but it is decomposable (in fact, it has a full exceptional collection).


\section{Griffiths components}

In this section we discuss the behavior of derived categories under standard birational transformations.
We start with a relative version of splitting off an exceptional object.

\subsection{Relative exceptional objects}

Assume $f:X \to Y$ is a morphism of smooth projective varieties.
Let $\CE \in \BD(X)$ be an object such that
\begin{equation}\label{define-relative-exceptional}
Rf_*\RCHom(\CE,\CE) \cong \CO_Y.
\end{equation} 
When $Y = \Spec(\kk)$, the above condition is just the definition of an exceptional object.

\begin{lemma}\label{lemma-splitoff-base}
If $\CE$ enjoys~\eqref{define-relative-exceptional} then the functor $\BD(Y) \to \BD(X)$, $F \mapsto \CE \lotimes Lf^*(F)$ is fully faithful
and gives a semiorthogonal decomposition
\begin{equation*}
\BD(X) = \langle \Ker Rf_*\RCHom(\CE,-), \CE\lotimes Lf^*(\BD(Y)) \rangle.
\end{equation*}
\end{lemma}
\begin{proof}
Indeed, $Rf_*\RCHom(\CE,-)$ is the right adjoint of $\CE\lotimes Lf^*(-)$, and since
\begin{equation*}
Rf_*\RCHom(\CE,\CE\lotimes Lf^*(F)) \cong 
Rf_*(\RCHom(\CE,\CE)\lotimes Lf^*(F)) \cong
Rf_*\RCHom(\CE,\CE) \lotimes F,
\end{equation*}
the condition~\eqref{define-relative-exceptional} and Lemma~\ref{lemma-admissible-sod} prove the result.
\end{proof}

\begin{example}
Let $X = \PP_Y(\CV) \xrightarrow{\ f\ } Y$ be the projectivization of a vector bundle $\CV$ of rank $r$ on $Y$. Then $Rf_*\CO_X \cong \CO_Y$,
hence any line bundle on $X$ satisfies~\eqref{define-relative-exceptional}. So, iterating the construction of Lemma~\ref{lemma-splitoff-base}
(along the lines of Example~\ref{example-pn})
we get Orlov's s.o.d.\ for the projectivization
\begin{equation}\label{equation-sod-projectivization}
\BD(\PP_Y(\CV)) = \langle \CO(1-r)\otimes Lf^*(\BD(Y)), \dots, \CO(-1)\otimes Lf^*(\BD(Y)), Lf^*(\BD(Y)) \rangle.
\end{equation}
\end{example}

\begin{example}
Let $\CQ \subset \PP_Y(\CV)$ be a flat family of quadrics with $f: X \to Y$ being the projection. 
Then similarly we get a semiorthogonal decomposition
\begin{equation}\label{equation-sod-quadric-bundle}
\BD(\CQ) = \langle \BD(Y,\Cl_0),\CO(3-r)\otimes Lf^*(\BD(Y)), \dots, \CO(-1)\otimes Lf^*(\BD(Y)), Lf^*(\BD(Y)) \rangle,
\end{equation}
where $\Cl_0$ is the sheaf of even parts of Clifford algebras on $Y$ associated with the family of quadrics $\CQ$
(see~\cite{K08} for details).
\end{example}

\subsection{Semiorthogonal decomposition of a blowup}

The most important for the birational geometry is the following s.o.d.
Let $X = \Bl_Z(Y)$ be the blowup of a scheme $Y$ in a locally complete intersection subscheme $Z \subset Y$ of codimension $c$. Then we have
the following blowup diagram
\begin{equation*}
\xymatrix{
X \ar[d]_f & E \ar[l]_i \ar@{=}[r] \ar[d]^p & \PP_Z(\CN_{Z/Y}) \\
Y & Z \ar[l]_j
}
\end{equation*}
where the exceptional divisor $E$ is isomorphic to the projectivization of the normal bundle, and its natural map to $Z$ 
is the standard projection of the projectivization. Note also that under this identification, the normal bundle $\CO_E(E)$ 
of the exceptional divisor is isomorphic to the Grothendieck line bundle $\CO_{E}(-1)$
\begin{equation}\label{equation-oe-e}
\CO_E(E) \cong \CO_{E}(-1)
\end{equation}
on the projectivization $E \cong \PP_Z(\CN_{Z/Y})$. We will use the powers of this line bundle 
to construct functors from $\BD(Z)$ to $\BD(X)$. 
For each $k \in \ZZ$ we consider the Fourier--Mukai functor with kernel $\CO_E(k)$, i.e.\
\begin{equation}
\Phi_{\CO_E(k)}(F) = Ri_*(\CO_{E}(k) \lotimes Lp^*(-)) \colon \BD(Z) \to \BD(X).
\end{equation} 

\begin{theorem}[Orlov's blowup formula]\label{theorem-sod-blowup}
The functor $\Phi_{\CO_E(k)} : \BD(Z) \to \BD(X)$ is fully faithful for each $k$ as well as the functor $Lf^*:\BD(Y) \to \BD(X)$. 
Moreover, they give the following semiorthogonal decomposition
\begin{equation}\label{equation-sod-blowup}
\BD(X) = \langle \Phi_{\CO_E(1-c)}(\BD(Z)), \dots, \Phi_{\CO_E(-1)}(\BD(Z)), Lf^*(\BD(Y)) \rangle.
\end{equation}
\end{theorem}
\begin{proof}[Sketch of proof]
First, $Rf_*\CO_X \cong \CO_Y$ by a local computation (locally we can assume that the ideal of $Z$
is generated by $c$ functions, this allows to embed $X$ explicitly into $Y \times \PP^{c-1}$
and to write down an explicit resolution for its structure sheaf; pushing it forward to $Y$ 
proves the claim). Hence by Lemma~\ref{lemma-splitoff-base} the functor $Lf^*$ is fully faithful.

Further, the right adjoint functor $\Phi^!_{\CO_E(k)}$ of $\Phi_{\CO_E(k)}$ is given by 
\begin{equation*}
F \mapsto 
Rp_*(\CO_{E}(-k) \lotimes i^!(F)) =
Rp_*(\CO_{E}(-k) \lotimes Li^*(F) \otimes \CO_E(E)[-1]) =
Rp_*(\CO_{E}(-k-1) \lotimes Li^*(F)[-1]),
\end{equation*}
and hence the composition $\Phi^!_{\CO_E(k)} \circ \Phi_{\CO_E(k)}$ is given by
\begin{equation*}
F \mapsto 
Rp_*(\CO_{E}(-k-1) \lotimes Li^*(Ri_*(\CO_{E}(k) \lotimes Lp^*(F))))[-1] \cong
Rp_*(\CO_{E}(-1) \lotimes Li^*(Ri_*(Lp^*(F))))[-1]. 
\end{equation*}
Note that $i$ is a divisorial embedding, hence the composition $Li^*\circ Ri_*$ comes wtih a distinguished triangle
\begin{equation*}
G \otimes \CO_E(1)[1] \to Li^*Ri_*(G) \to G \to G \otimes \CO_E(1)[2]
\end{equation*}
(this is a derived category version of the ``fundamental local isomorphism'' $\Tor_p(i_*G,i_*\CO_E) \cong G \otimes \Lambda^p\CN^\vee_{E/X}$ 
together with the isomorphism $\CN^\vee_{E/X} \cong \CO_E(-E) \cong \CO_E(1)$, see~\cite[Cor.~11.4(ii)]{H06}). Therefore we have a distinguished triangle
\begin{equation*}
Rp_*(Lp^*(F)) \to \Phi^!_{\CO_E(k)}\Phi_{\CO_E(k)}(F) \to Rp_*(\CO_E(-1) \lotimes Lp^*(F))[-1] \to Rp_*(Lp^*(F))[1]. 
\end{equation*}
Using the projection formula and the fact that $Rp_*\CO_E \cong \CO_Z$, $Rp_*\CO_E(-1) = 0$, we conclude that $\Phi^!_{\CO_E(k)}\Phi_{\CO_E(k)}(F) \cong F$,
hence $\Phi_{\CO_E(k)}$ is fully faithful.

Analogously, computing the composition $\Phi^!_{\CO_E(k)}\Phi_{\CO_E(l)}$ we get a distinguished triangle
\begin{equation*}
Rp_*(\CO_E(l-k) \lotimes Lp^*(F)) \to \Phi^!_{\CO_E(k)}\Phi_{\CO_E(l)}(F) \to Rp_*(\CO_E(l-k-1) \lotimes Lp^*(F))[-1] \to \dots
\end{equation*}
Recall that $Rp_*\CO_E(-t) = 0$ for $1 \le t \le c-1$. 
As for $1 - c \le l < k \le -1$
we have $1-c \le l-k,\ l-k-1 \le -1$, hence we have $\Phi^!_{\CO_E(k)}\circ\Phi_{\CO_E(l)} = 0$. 
This shows that the first $c-1$ components of~\eqref{equation-sod-blowup} are semiorthogonal. 

Finally, for the composition $Rf_*\circ \Phi_{\CO_E(k)}$ we have
\begin{equation*}
Rf_*\circ \Phi_{\CO_E(k)}(F) = 
Rf_*Ri_*(\CO_E(k)\lotimes Lp^*(F)) \cong
Rj_*Rp_*(\CO_E(k)\lotimes Lp^*(F)) \cong
Rj_*(Rp_*\CO_E(k) \lotimes F),
\end{equation*}
and since $Rp_*\CO_E(k) = 0$ for $1-c \le k \le -1$, it follows that the composition is zero, and hence the last component 
of~\eqref{equation-sod-blowup} is semiorthogonal to the others. 

It remains to show that the components we just described generate the whole category $\BD(X)$. This computation is slightly
technical. Basically one can compute the composition $Lf^*\circ Rj_*$ and show that it equals $Ri_*\circ Lp^*$ modulo
the first $c-1$ components of~\eqref{equation-sod-blowup}. This means that the RHS of~\eqref{equation-sod-blowup} contains 
the subcategory $\Phi_{\CO_E}(\BD(Z))$. By~\eqref{equation-sod-projectivization}, it then follows that $Ri_*(\BD(E))$
is contained in the RHS, hence any object in the orthogonal is in the kernel of $Li^*$, hence is supported on $X \setminus E$.
But $f$ defines an isomorphism $X \setminus E \cong Y \setminus Z$, hence any such object $F$ can be written as
$Lf^*(Rf_*(F))$, hence still lives in the RHS of~\eqref{equation-sod-blowup}. Altogether, this argument proves the Theorem.
\end{proof}

Roughly speaking, we can interpret Theorem~\ref{theorem-sod-blowup} by saying that the ``difference'' between the derived categories of $X$ and $Y$ is given by a number
of derived categories of subvarieties of dimension $\le n-2$, where $n = \dim X = \dim Y$. On the other hand, by Weak Factorization
Theorem any two birational varieties can be linked by a sequence of blowups and blowdowns with smooth (and hence lci) centers.
Thus one can get the resulting derived category from the original one by an iterated ``addition'' and ``subtraction''
of derived categories of smooth projective varieties of dimension at most $n-2$. 

\subsection{Griffiths components}

The above observation suggests the following series of definitions.

\begin{definition}
Define the {\sf geometric dimension} of a triangulated category $\CA$, $\gdim(\CA)$, as the minimal integer $k$
such that $\CA$ can be realized as an admissible subcategory of a smooth projective and connected variety of dimension $k$.
\end{definition}

\begin{example}\label{example-gdim-0}
If $\CA$ is a triangulated category of geometric dimension $0$ then $\CA \cong \BD(\kk)$.
Indeed, by definition $\CA$ should be an admissible subcategory of the derived category of a smooth projective
connected variety of dimension 0, i.e., of $\BD(\Spec(\kk)) = \BD(\kk)$. But this category
is indecomposable (see Corollary~\ref{corollary-k-indecomposable}), hence $\CA \cong \BD(\kk)$.
\end{example}

\begin{example}\label{example-gdim-1}
If $\CA$ is an indecomposable triangulated category of geometric dimension $1$ then $\CA \cong \BD(C)$, 
where $C$ is a curve of genus $g \ge 1$. Indeed, by definition $\CA$ should be an admissible subcategory 
of the derived category of a smooth projective connected variety of dimension 1, i.e., of $\BD(C)$. 
If $g(C) \ge 1$ then $\BD(C)$ is indecomposable by Proposition~\ref{proposition-curve-indecomposable} and so $\CA = \BD(C)$. 
If $g(C) = 0$, i.e., $C \cong \PP^1$, then again by Proposition~\ref{proposition-curve-indecomposable} 
any nontrivial decomposition of $\BD(C)$ consists of two exceptional objects, so if $\CA \subset \BD(C)$ 
is its indecomposable admissible subcategory then $\CA \cong \BD(\kk)$, but then its geometric dimension is 0.
\end{example}

A classification of triangulated categories of higher geometrical dimension (if possible)
should be much more complicated. For instance, it was found out recently that some surfaces
of general type with $p_g = q = 0$ contain admissible subcategories with zero Hochschild homology
(so called quasiphantom categories). These categories are highly nontrivial examples of categories
of geometric dimension 2.

\begin{definition}
A semiorthogonal decomposition $\CT = \langle \CA_1,\dots,\CA_m \rangle$ is {\sf maximal}, if each component $\CA_i$
is an indecomposable category, i.e., does not admit a nontrivial semiorthogonal decomposition.
\end{definition}

\begin{definition}
Let $\BD(X) = \langle \CA_1,\dots,\CA_m \rangle$ be a maximal semiorthogonal decomposition. 
Let us say that its component $\CA_i$ is a {\sf Griffiths component}, if 
\begin{equation*}
\gdim(\CA_i) \ge \dim X - 1.
\end{equation*}
We denote by
\begin{equation*}
\Griff(X) := \{ \CA_i \mid \gdim(\CA_i) \ge \dim X - 1 \}
\end{equation*}
the set of Griffiths components of a given maximal semiorthogonal decomposition of $\BD(X)$.
\end{definition}

One could hope that the set of Griffiths components $\Griff(X)$ is well defined (does not depend on the choice of a maximal
semiorthogonal decomposition). Let us imagine this is the case and try to discuss rationality of Fano 3-folds 
on this basis. First, note the following

\begin{lemma}
If the set $\Griff(X)$ of Griffiths components is well defined then it is a birational invariant.
\end{lemma}
\begin{proof}
Recall that by Weak Factorization Theorem if $X$ and $Y$ are smooth projective varieties birational to each other
then there is a sequence of blowups and blowdowns with smooth centers connecting $X$ and~$Y$. So, to show that
the set of Griffiths components is a birational invariant, it is enough to check that it does not change under 
a smooth blowup. So, assume that $X = \Bl_Z(Y)$. Let $\BD(Y) = \langle \CA_1,\dots,\CA_m \rangle$ and
$\BD(Z) = \langle \CB_1, \dots, \CB_k \rangle$ be maximal semiorthogonal decompositions. Then by Theorem~\ref{theorem-sod-blowup}
we have a semiorthogonal decomposition of $\BD(X)$ with components $\CA_i$ and $\CB_j$ (repeated $c-1$ times, where
$c$ is the codimension of $Z$). Note also, that $\CB_j$ is an admissible subcategory of $\BD(Z)$, hence
\begin{equation*}
\gdim(\CB_j) \le \dim Z \le \dim X - 2,
\end{equation*}
hence none of $\CB_j$ is a Griffiths component of $\BD(X)$. Thus the sets of Griffiths components of $\BD(X)$ and $\BD(Y)$ coincide.
\end{proof}

\begin{corollary}\label{corollary-rational-griffiths}
If the set $\Griff(X)$ of Griffiths components is well defined then any rational variety of dimension at least $2$ has no Griffiths components.
\end{corollary}
\begin{proof}
By previous Lemma it is enough to show that $\PP^n$ has no Griffiths components. But by Example~\ref{example-pn}
there is a semiorthogonal decomposition of $\BD(\PP^n)$ with all components being derived categories of points.
Their geometrical dimension is zero, so as soon as $n \ge 2$, none of them is a Griffiths component.
\end{proof}

This ``birational invariant'' can be effectively used.
Consider, for example, a smooth cubic 3-fold $V_3$. As it was discussed in section~\ref{ssection-fano3},
it has a semiorthogonal decomposition with two exceptional objects and a category $\CA_{V_3}$ as components.

\begin{proposition}
If the set $\Griff(X)$ of Griffiths components is well defined then for a smooth cubic $3$-fold it is nonempty.
In particular, a smooth cubic $3$-fold is not rational.
\end{proposition}
\begin{proof}
If the set of Griffiths components of a smooth cubic 3-fold would be empty, then the category $\CA := \CA_{V_3}$ should have
a semiorthogonal decomposition with components of geometrical dimension at most 1, i.e.\ by Examples~\ref{example-gdim-0}
and~\ref{example-gdim-1} with components $\BD(\kk)$ and $\BD(C_i)$ with $g(C_i) \ge 1$. To show this is not true we will use
Hochschild homology. 

Recall that by HKR isomorphism the Hochschild homology of $\BD(V_3)$ can be computed in terms of its Hodge numbers.
On the other hand, by Griffiths residue Theorem the Hodge diamond of $V_3$ looks as
\begin{equation*}
{\tiny\xymatrix@=-1ex{
&&& 1 \\
&& 0 && 0 \\
& 0 && 1 && 0 \\
0 && 5 && 5 && 0\\
& 0 && 1 && 0 \\
&& 0 && 0 \\
&&& 1 
}}
\end{equation*}
hence $\HOH_\bullet(\BD(V_3)) = \kk^5[1] \oplus \kk^4 \oplus \kk^5[-1]$. To obtain the category $\CA_{V_3}$ we split off
two exceptional bundles, hence by additivity of Hochschild homology (Theorem~\ref{theorem-hoh-additivity}) we get 
\begin{equation*}
\HOH_\bullet(\CA_{V_3}) \cong \kk^5[1] \oplus \kk^2 \oplus \kk^5[-1].
\end{equation*}
Now assume $\CA_{V_3}$ has a semiorthogonal decomposition with $m$ components equivalent to $\BD(\kk)$ 
and $k$ components equivalent to $\BD(C_1)$, \dots, $\BD(C_k)$. Then by additivity of Hochschild
homology and Example~\ref{example-hoh-dim-0-1} we get
%
\begin{equation*}
\sum_{i=1}^k g(C_i) = \dim \HOH_1(\CA) = 5,
\qquad
m + 2k = \dim \HOH_0(\CA) = 2.
\end{equation*}
It follows from the first that $k \ge 1$ and from the second $k \le 1$. Hence $k = 1$, $g(C_1) = 5$, and $m = 0$.
So, the only possibility is if $\CA \cong \BD(C)$ with $g(C) = 5$. But $\SS_{\CA}^3 \cong [5]$, while $\SS_{\BD(C)}$
clearly does not have this property. 
\end{proof}

The same argument works for $X_{14}$, $X_{10}$, $dS_4$, $V_{2,3}$, $V_6^{1,1,1,2,3}$, $V_4$, and $dS_6$, i.e.\ for all prime Fano 3-folds
with the exception of those which are known to be rational, and $V_{2,2,2}$. For the last one we need another way to check that
an equivalence $\CA_{V_{2,2,2}} \cong \BD(C)$ is impossible. 
One of the possibilities is by comparing the Hochschild cohomology of these categories.

\begin{remark}
Most probably, the category $\CA_{V_3}$ (as well as the other similar categories) is indecomposable. However,
this is not so easy to prove, but even without this the above argument works fine.
\end{remark}

\subsection{Bad news}\label{section-bad-news}

Unfortunately, Griffiths components are not well defined. To show this we will exhibit a contradiction 
with Corollary~\ref{corollary-rational-griffiths} by constructing an admissible subcategory of geometric 
dimension greater than 1 in a rational threefold. The construction is based on the following interesting 
category discovered by Alexei Bondal in 90's.

%
%
%
%
%

Consider the following quiver with relations
\begin{equation}\label{bq}
Q = \left(\left. \xymatrix@C=3em{ 
\bullet \ar@<1ex>[r]^{\alpha_1} \ar@<-1ex>[r]_{\alpha_2} & 
\bullet \ar@<1ex>[r]^{\beta_1} \ar@<-1ex>[r]_{\beta_2} & 
\bullet 
} \right|\
\beta_1\alpha_2 = \beta_2\alpha_1 = 0
\right).
\end{equation}
As any oriented quiver it has a full exceptional collection
\begin{equation}\label{equation-sod-bondal-quiver}
\BD(Q) = \langle P_1,P_2,P_3 \rangle
\end{equation}
with $P_i$ being the projective module of the $i$-th vertex. On the other hand,
it has an exceptional object
\begin{equation}\label{oe}
E = \left( \xymatrix@C=3em{ 
\kk \ar@<1ex>[r]^{1} \ar@<-1ex>[r]_{0} & 
\kk \ar@<1ex>[r]^{1} \ar@<-1ex>[r]_{0} & 
\kk
} \right).
\end{equation}
It gives a semiorthogonal decomposition 
\begin{equation*}
\BD(Q) = \langle \CA_0, E \rangle,
\end{equation*}
and it is interesting that the category $\CA_0 := E^\perp$ has no exceptional objects. Indeed, a straightforward computation shows
that the Euler form on $K_0(\CA_0)$ is skew-symmetric, and so in $K_0(\CA_0)$ there are no vectors of square 1.

In particular, any indecomposable admissible subcategory of $\CA_0$ has geometric dimension greater than~1.
Indeed, by Example~\ref{example-gdim-0} and the above argument we know that $\CA_0$ has no admissible subcategories
of geometric dimension 0. Furthermore, if it would have an admissible subcategory of geometric dimension 1, then by~\eqref{example-gdim-1}
this subcategory should be equivalent to $\BD(C)$ with $g(C) \ge 1$. But then by additivity of Hochschild homology 
$\dim\HOH_1(\CA_0) \ge \dim\HOH(\BD(C)) \ge g(C) \ge 1$, 
which contradicts $\HOH_1(\BD(Q)) = 0$ again by additivity and by~\eqref{equation-sod-bondal-quiver}.

It remains to note that by~\cite{K13} there is a rational 3-fold $X$ with $\BD(Q)$ (and a fortiori $\BD(\CA_0)$) being a semiorthogonal
component of $\BD(X)$. 

\begin{remark}
The same example shows that semiorthogonal decompositions do not satsfy Jordan--H\"older property.
\end{remark}

However, as one can see on the example of prime Fano 3-folds, there is a clear correlation between
appearance of nontrivial pieces in derived categories and nonrationality. Most probably, the notion
of a Griffiths component should be redefined in some way, and then it will be a birational invariant.
One of the possibilities is to include fractional Calabi--Yau property into the definition.
Another possibility is to consider minimal categorical resolutions of singular varieties
as geometric categories and thus redefine the notion of the geometric dimension.
Indeed, as we will see now the category $\CA_0$ is of this nature.

\subsection{Minimal categorical resolution of a nodal curve}

The category $\CA_0$ from the previous section has a nice geometric interpretation.
It is a minimal categorical resolution of singularities of a rational nodal curve.

Let $C_0 := \{ y^2z = x^2(x-z) \} \subset \PP^2$ be a rational nodal curve and denote by $\bp_0 \in C_0$ its node. The normalization map
$C \cong \PP^1 \xrightarrow{\ \nu\ } C_0$ is not a ``categorical resolution'', since $\nu_*\CO_C \not\cong \CO_{C_0}$
and hence the pullback functor $L\nu^*$ from the derived category of $C_0$ to the derived category of $C$ is not fully faithful.
However, following the recipe of~\cite{KL}, one can construct a categorical resolution of $C_0$ by gluing $\BD(C)$ with $\BD(\bp_0) \cong \BD(\kk)$
along $\nu^{-1}(\bp_0) = \{ \bp_1, \bp_2 \}$. The resulting category is equivalent to the derived category $\BD(Q)$ of Bondal's quiver.

Indeed, the subcategory of $\BD(Q)$ generated by the first two vertices (or, equivalently, by the first two projective modules $P_1$ and $P_2$)
of the quiver is equivalent to $\BD(\PP^1) \cong \BD(C)$, so that the two arrows $\alpha_1$ and $\alpha_2$ between the first two vertices correspond to the homogeneous
coordinates $(x_1:x_2)$ on $\PP^1$. The third vertex gives $\BD(\bp_0)$, and the two arrows between the second and the third
vertex of $Q$ correspond to the points $\bp_1$ and~$\bp_2$. Finally, the relations of the quiver follow from 
$x_1(\bp_2) = x_2(\bp_1) = 0$, the vanishing of (suitably chosen) homogeneous coordinates of $\PP^1$ at points $\bp_1$ and $\bp_2$.

The resolution functor takes an object $F \in \BD(C_0)$ to the representation of $Q$ defined by
\begin{equation*}
(H^\bullet(C,L\nu^*F\otimes\CO_C(-1)), H^\bullet(C,L\nu^*F), Li_0^*F),
\end{equation*}
where $i_0 : \bp_0 \to C_0$ is the embedding of the node. It is easy to see that its image is contained
in the subcategory $\CA_0 = P^\perp \subset \BD(Q)$, so the latter can be considered as a categorical resolution of $C_0$ as well.

If we consider the category $\CA_0$ as a geometric category of dimension 1, so that $\gdim(\CA_0) = 1$, then $\CA_0$
is no longer a Griffiths component of the rational threefold $X$ from section~\ref{section-bad-news}, and this is no longer
a counterexample.


\section{Higher dimensional varieties}

Probably, the most interesting (from the birational point of view) example of a 4-dimensional variety is a cubic fourfold.
There are examples of cubic fourfolds which are known to be rational, but general cubic fourfolds are expected to be nonrational.
In this section we will discuss how does rationality of cubic fourfolds correlate with the structure of their derived categories.

\subsection{A noncommutative K3 surface associated with a cubic fourfold}

A cubic fourfold is a hypersurface $X \subset \PP^5$ of degree 3. By adjunction we have
\begin{equation*}
K_X = -3H,
\end{equation*}
where $H$ is the class of a hyperplane in $\PP^5$. So, $X$ is a Fano 4-fold of index 3.
Therefore, by Example~\ref{example-fano-index-r} we have a semiorthogonal decomposition
\begin{equation}\label{equation-sod-cubic-4}
\BD(X) = \langle \CA_X, \CO_X(-2H), \CO_X(-H), \CO_X \rangle.
\end{equation}
In fact, this is a maximal semiorthogonal decomposition.

\begin{proposition}
The category $\CA_X$ is a connected Calabi--Yau category of dimension $2$ with Hochschild homology
isomorphic to that of K$3$ surfaces. In particular, $\CA_X$ is indecomposable
and~\eqref{equation-sod-cubic-4} is a maximal semiorthogonal decomposition.
\end{proposition}
\begin{proof}
The proof of the fact that $\SS_{\CA_X} \cong [2]$ is a bit technical, it can be found in~\cite{K04} and~\cite{K15}.
The Hochschild homology computation, on a contrary, is quite simple. 
By Lefschetz hyperplane Theorem and Griffiths residue Theorem, the Hodge diamond of $X$ looks as
\begin{equation*}
\tiny{\xymatrix@=-1ex{
&&&& 1 \\
&&& 0 && 0 \\
&& 0 && 1 && 0 \\
& 0 && 0 && 0 && 0 \\
0 && 1 && 21 && 1 && 0 \\
& 0 && 0 && 0 && 0 \\
&& 0 && 1 && 0 \\
&&& 0 && 0 \\
&&&& 1 
}}
\end{equation*}
%
%
Thus by HKR isomorphism we have $\HOH_\bullet(\BD(X)) = \kk[2] \oplus \kk^{25} \oplus \kk[-2]$.
Since $\CA_X$ is the orthogonal complement of an exceptional triple in the category $\BD(X)$, 
by additivity of Hochschild homology (Theorem~\ref{theorem-hoh-additivity}) it follows that
\begin{equation*}
\HOH_\bullet(\CA_X) = \kk[2] \oplus \kk^{22} \oplus \kk[-2].
\end{equation*}
By HKR isomorphism this coincides with the dimensions of Hochschild homology of K3 surfaces.

Since $\CA_X$ is a 2-Calabi--Yau category there is an isomorphism $\HOH_i(\CA_X) \cong \HOH^{i+2}(\CA_X)$. Therefore,
from the above description of Hochschild homology it follows that $\HOH^0(\CA_X)$ is one-dimensional, i.e., the category $\CA_X$ is connected
(alternatively, $\HOH^0(\CA_X)$ can be easily computed via the technique of~\cite{K12}). 

Indecomposability of $\CA_X$ then follows from Proposition~\ref{proposition-cy-indecomposable}.
The components generated by exceptional objects are indecomposable by Corollary~\ref{corollary-k-indecomposable}.
\end{proof}

Being Calabi--Yau category of dimension 2, the nontrivial component $\CA_X$ of $\BD(X)$ can be considered as 
a noncommutative K3 surface. As we will see soon, for some special cubic fourfolds $X$ there are equivalences
$\CA_X \cong \BD(S)$ for appropriate K3 surfaces. Thus $\CA_X$ is indeed a deformation of $\BD(S)$ for a K3 surface~$S$.

Clearly, if $\CA_X \cong \BD(S)$ then $\gdim(\CA_X) = 2$ and pretending that the set of Griffiths components $\Griff(X)$
is well defined, it follows that it is empty.
On the other hand, if $\CA_X \not\cong \BD(S)$ then $\gdim(\CA_X) > 2$ by Conjecture~\ref{conjecture-cydim}, 
and so $\CA_X$ is a Griffiths component.
This motivates the following

\begin{conjecture}\label{conjecture-cubic4}
A cubic fourfold $X$ is rational if and only if there is a smooth projective K$3$ surface $S$ and an equivalence $\CA_X \cong \BD(S)$.
\end{conjecture}

We will show that this Conjecture agrees perfectly with the known cases of rational cubic fourfolds.

\subsection{Pfaffian cubics}

Consider a $6\times 6$ skew-symmetric matrix with entries being linear forms on $\PP^5$.
Its determinant is the square of a degree 3 homogeneous polynomial in coordinates on $\PP^5$,
called the Pfaffian of the matrix. The corresponding cubic hypersurface is called a {\sf pfaffian cubic}.

In other words, if $\PP^5 = \PP(V)$ with $\dim V = 6$, then a cubic hypersurface $X \subset \PP(V)$ is pfaffian
if there is another 6-dimensional vector space $W$ and a linear map $\varphi:V \to \Lambda^2W^\vee$ such that
$X = \varphi^{-1}(\Pf)$, where $\Pf \subset \PP(\Lambda^2W^\vee)$ is the locus of degenerate skew-forms.
In what follows we will assume that $X$ is not a cone (hence $\varphi$ is an embedding), and that the subspace $\varphi(V) \subset \Lambda^2W^\vee$ 
does not contain skew-forms of rank $2$ (i.e., $\PP(V) \cap \Gr(2,W^\vee) = \emptyset$). To unburden notation we consider
$V$ as a subspace of $\Lambda^2W^\vee$.

\begin{remark}\label{remark-pfaffian-delpezzo}
There is a geometric way to characterize pfaffian cubics. First, assume a cubic $X$ is pfaffian and is given
by a subspace $V \subset \Lambda^2W^\vee$. Choose a generic hyperplane $W_5 \subset W$ and consider 
the set $R_{W_5} \subset X$ of all degenerate skew-forms in $V$ the kernel of which is contained in $W_5$. 
It is easy to see that the kernel of a rank 4 skew-form is contained in $W_5$ if and only if its restriction to $W_5$ is decomposable.
Note also that such restriction
is never zero (since we assumed there are no rank 2 forms in $V$), hence the composition
\begin{equation*}
\varphi_{W_5}: V \hookrightarrow \Lambda^2W^\vee \twoheadrightarrow \Lambda^2W_5^\vee
\end{equation*}
is an embedding.
Therefore, $R_{W_5} = \varphi_{W_5}(\PP(V)) \cap \Gr(2,W_5)$.
For generic choice of $W_5$ this intersection is dimensionally transverse and smooth, hence 
$R_{W_5}$ is a quintic del Pezzo surface in $X$.
Vice versa, a cubic fourfold $X$ containing a quintic del Pezzo surface is pfaffian (see~\cite{Bea00}).
\end{remark}

If $X$ is a pfaffian cubic fourfold, one can associate with $X$ a K3 surface as follows. 
Consider the annihilator $V^\perp \subset \Lambda^2W$ of $V \subset \Lambda^2W^\vee$ and define
\begin{equation}\label{equation-pfaffian-k3}
S = \Gr(2,W) \cap \PP(V^\perp).
\end{equation}
One can check the following

\begin{lemma}
A pfaffian cubic $X$ is smooth if and only if the corresponding intersection~\eqref{equation-pfaffian-k3} is dimensionally transverse and smooth,
i.e.\ $S$ is a K$3$ surface.
\end{lemma}

Any pfaffian cubic is rational. One of the ways to show this is the following.

By definition, the points of $X \subset \PP(V)$ correspond to skew-forms on $W$ of rank 4. Each such form
has a two-dimensional kernel space $K \subset W$. Considered in a family, they form a rank 2 vector subbundle
$\CK \subset W \otimes \CO_X$. Its projectivization $\PP_X(\CK)$ comes with a canonical map to $\PP(W)$.
On the other hand, the K3 surface $S$ defined by~\eqref{equation-pfaffian-k3} is embedded into $\Gr(2,W)$,
hence carries the restriction of the rank 2 tautological bundle of the Grassmanian, i.e.\ a subbundle 
$\CU \subset W \otimes \CO_S$. Its projectivization $\PP_S(\CU)$ also comes with a canonical map to $\PP(W)$.

\begin{proposition}
The map $p:\PP_X(\CK) \to \PP(W)$ is birational and the indeterminacy locus of the inverse map
coincides with the image of the map $q:\PP_S(\CU) \to \PP(W)$.
\end{proposition}
\begin{proof}
The proof is straightforward. Given $w \in \PP(W)$ its preimage in $\PP_X(\CK)$ is the set of all skew-forms in $V$
containing $w$ in the kernel, hence it is the intersection $\PP(V) \cap \PP(\Lambda^2w^\perp) \subset \PP(\Lambda^2W^\vee)$.
Typically it is a point, and if it is not a point then the map 
\begin{equation*}
V \hookrightarrow \Lambda^2W^\vee \stackrel{w}\twoheadrightarrow w^\perp \subset W^\vee
\end{equation*}
has at least two-dimensional kernel, i.e.\ at most 4-dimensional image. If the image is contained in the subspace $U^\perp \subset w^\perp$
for a 2-dimensional subspace $U \subset W$, then $U$ gives a point on $S$ and $(U,w)$ gives a point on $\PP_S(\CU)$
which projects by $q$ to $w$.
\end{proof}

\begin{remark}
One can show that if the K3 surface $S \subset \Gr(2,W) \subset \PP(\Lambda^2W)$ contains no lines,
then the map $q:\PP_S(\CU) \to \PP(W)$ is a closed embedding. If, however, $S$ has a line $L$
then $\CU_{|L} \cong \CO_L \oplus \CO_L(-1)$ and $\PP_L(\CU_{|L}) \subset \PP_S(\CU)$
is the Hirzebruch surface $F_1$. Denote by $\tilde{L} \subset \PP_L(\CU_{|L})$
its exceptional section. Then the map $q$ contracts $\tilde{L}$ to a point, and does this
for each line in $S$.
\end{remark}

Consider the obtained diagram
\begin{equation*}
\xymatrix{
& \PP_X(\CK) \ar[dl] \ar[dr]^p && \PP_S(\CU) \ar[dl]_q \ar[dr] \\
X && \PP(W) && S,
}
\end{equation*}
choose a generic hyperplane $W_5 \subset W$, and consider the ``induced diagram''
\begin{equation*}
\xymatrix{
& X' := p^{-1}(\PP(W_5)) \ar[dl] \ar[dr]^p && q^{-1}(\PP(W_5)) =: S' \ar[dl]_q \ar[dr] \\
X && \PP(W_5) && S,
}
\end{equation*}
Then it is easy to see that (for general choice of the hyperplane $W_5$) the map $X' \to X$ is the blowup with center in the quintic
del Pezzo surface $R_{W_5}$ (see Remark~\ref{remark-pfaffian-delpezzo}),
and the map $S' \to S$ is the blowup with center in $R'_{W_5} := S \cap \Gr(2,W_5) = \PP(V^\perp) \cap \Gr(2,W_5)$,
which is a codimension 6 linear section of $\Gr(2,5)$, i.e.\ just 5 points. Finally, the map $q:S' \to \PP(W_5)$ is a closed embedding
and the map $p$ is the blowup with center in $q(S') \subset \PP(W_5)$. Thus, the blowup of $X$ in a quintic del Pezzo surface is isomorphic
to the blowup of $\PP(W_5) = \PP^4$ in the surface $S'$, isomorphic to the blowup of the K3 surface $S$ in 5 points.
In particular, $X$ is rational.

To show the implication for the derived category one needs more work.

\begin{theorem}\label{theorem-pfaffian-equivalence}
If $X$ is a pfaffian cubic fourfold and $S$ is the K$3$ surface defined by~\eqref{equation-pfaffian-k3},
then $\CA_X \cong \BD(S)$.
\end{theorem}

A straightforward way to prove this is by considering the two semiorthogonal decompositions of $X'$:
\begin{equation*}
\BD(X') = \langle \BD(R_{W_5}), \CA_X,\CO_{{X'}}(-2H),\CO_{{X'}}(-H),\CO_{{X'}} \rangle,
\qquad
\BD(X') = \langle \BD(R'_{W_5}), \BD(S), \BD(\PP(W_5)) \rangle.
\end{equation*}
Taking into account that $\BD(R_{W_5})$ is generated by an exceptional collection of length 7 (since $R_{W_5}$
is isomorphic to the blowup of $\PP^2$ in 4 points), we see that the first s.o.d.\
consists of 10 exceptional objects and the category $\CA_X$, while the second s.o.d.\ consists of 10 exceptional objects
and $\BD(S)$. So, if we would have a Jordan--H\"older property, it would follow that $\CA_X \cong \BD(S)$.
As the property is wrong, we should instead, find a sequence of mutations taking one s.o.d.\ to the other.

There is also a completely different proof of Theorem~\ref{theorem-pfaffian-equivalence} based on homological projective duality
(see~\cite{K07,K06b}).

\begin{remark}
In fact, the locus of pfaffian cubic fourfolds is a dense open subset in the divisor $\CC_{14}$ in the moduli space
of all cubic fourfolds, which consists of cubic fourfolds containing a quartic scroll. 
\end{remark}

\subsection{Cubics with a plane}

Another very interesting class of cubic fourfolds is formed by cubic fourfolds $X \subset \PP(V)$ containing a plane.
So, let $U_3 \subset V$ be a 3-dimensional subspace and $W_3 := V/U_3$. Assume that $X$ contains $\PP(U_3)$.
Then the linear projection from $\PP(U_3)$ defines a rational map $X \dashrightarrow \PP(W_3)$.
To resolve it we have to blowup the plane $\PP(U_3)$ first. We will get the following diagram
\begin{equation*}
\xymatrix{
& \TX \ar[dl]_\pi \ar[dr]^p \\
X && \PP(W_3)
}
\end{equation*}
The map $p$ here is a fibration with fibers being 2-dimensional quadrics, and the discriminant $D \subset \PP(W_3)$ being a plane sextic curve.
Assume for simplicity that $X$ is sufficiently general, so that $D$ is smooth. Then the double covering 
\begin{equation}\label{equation-plane-k3}
S \to \PP(W_3)
\end{equation} 
ramified in $D$ is a smooth K3 surface (and if $D$ is mildly singular, which is always true for smooth $X$, then one should
take $S$ to be the minimal resolution of singularities of the double covering). Then $S$ comes with a natural Severi--Brauer variety,
defined as follows.

\begin{lemma}
Let $M$ be the Hilbert scheme of lines in the fibers of the morphism $p$. Then the natural morphism $M \to \PP(W_3)$
factors as the composition $M \to S \to \PP(W_3)$ of a $\PP^1$-fibration $q:M \to S$, followed by the double covering $S \to \PP(W_3)$.
\end{lemma}
\begin{proof}
By genericity assumption all fibers of $p$ are either nondegenerate quadrics or quadrics of corank~1.
For a nondegenerate fiber, the Hilbert scheme of lines is a disjoint union of two smooth conics. 
And for a degenerate fiber, it is just one smooth conic. Thus the Stein factorization of the map $M \to \PP(W_3)$ is a composition
of a nondegenerate conic bundle (i.e.\ a $\PP^1$-fibration) with a double covering. Moreover, the double
covering is ramified precisely in $D$, hence coincides with $S$.
\end{proof}

The constructed Severi--Brauer variety $q:M \to S$ corresponds to an Azumaya algebra on $S$, which we denote $\CB_0$.
One can show that its pushforward to $\PP(W_3)$ coincides with the sheaf of even parts of Clifford
algebras of the quadric fibration $p:\TX \to \PP(W_3)$. The derived category $\BD(S,\CB_0)$ of
sheaves of coherent $\CB_0$-modules on~$S$ is known as a ``twisted derived category'' of $S$, or 
as a ``twisted K3 surface''.

\begin{proposition}
If $X$ is a general cubic fourfold with a plane then $\CA_X \cong \BD(S,\CB_0)$.
\end{proposition}
\begin{proof}
Again, one can consider two semiorthogonal decompositions of $\BD(\TX)$:
\begin{align*}
\BD(\TX) = &\langle \BD(\PP(U_3)), \CA_X, \CO_\TX(-2H), \CO_\TX(-H), \CO_\TX \rangle,\\
\BD(\TX) = &\langle \BD(S,\CB_0), \BD(\PP(W_3)) \otimes \CO_\TX(-H), \BD(\PP(W_3)) \rangle,
\end{align*}
the first consists of 6 exceptional objects and the category $\CA_X$, the second again consists of 6 exceptional objects
and the category $\BD(S,\CB_0)$. An appropriate sequence of mutations (see~\cite{K10}) then identifies the categories $\CA_X$ and $\BD(S,\CB_0)$.
\end{proof}

Finally, the following simple argument shows that the quadric fibration $\tilde{X} \to \PP(W_3)$
is rational if and only if the twisting $\CB_0$ of the K3 surface $S$ is trivial.

\begin{theorem}
The quadric fibration $\TX \to \PP(W_3)$ is rational over $\PP(W_3)$ if and only if the Azumaya algebra $\CB_0$ on $S$ is Morita trivial.
\end{theorem}
\begin{proof}
Indeed, the first is equivalent to existence of a rational multisection of $p:\TX \to \PP(W_3)$ of odd degree, and the second is equivalent
to existence of a rational multisection of $q:M \to S$ of odd degree. It remains to note that given a rational multisection of $p$ of degree $d$
and considering lines in the fibers of $p$ intersecting it, we obtain a rational multisection of $q$ also of degree $d$.
Vice versa, given a rational multisection of $q$ of degree~$d$ (i.e. $d$ vertical and $d$ horizontal lines in a general fiber of $p$)
and considering the points of intersection of these vertical and horizontal lines, we obtain a rational multisection of $p$
of degree $d^2$. 
\end{proof}

\subsection{Final remarks}

We conclude this section with a short (and incomplete) list of results and papers touching related subjects.

First, there is a notion of a ``categorical representability'' introduced by M.~Bernardara
and M.~Bolognesi, see~\cite{BerBo}. It is designed for the same purpose as the notion 
of a Griffiths component, but in a slightly different way. 

Second, historically the first conjectural characterization of rationality of cubic fourfolds
was suggested by B.~Hassett in~\cite{Has} in terms of Hodge structures, see also a discussion
in~\cite{Ad}. The relation between the Conjecture of Hassett and Conjecture~\ref{conjecture-cubic4} 
was discussed in~\cite{AdTh}. A general discussion of the category $\CA_X$ for a cubic fourfold
can be found in~\cite{H15}.

Finally, it is worth noting that there are two more (classes of) varieties which behave very 
similarly to cubic 4-folds. The first are the Fano fourfolds of index 2 and degree 10 sometimes 
called Gushel--Mukai fourfolds. Their birational properties are discussed in~\cite{DIM} and~\cite{DK},
and their derived categories in~\cite{KP}. The second class is formed by the so-called K\"uchle
varieties of type (c5), see~\cite{Kuc}. It was shown in {\em loc.\ cit.}\/ that these varieties
have a Hodge structure of a K3 surface in the middle cohomology, and in~\cite{K14} it was conjectured
that their derived categories contain a noncommutative K3 category. However, so far nothing is known
about their birational geometry.

\end{document}